\documentclass{amsproc}
\usepackage[margin=1in]{geometry} 
\usepackage{amsmath,amsthm,amssymb}
 \usepackage{pifont}
\usepackage{mathrsfs}
\usepackage[all,cmtip]{xy}
\usepackage{graphics,graphicx}

\usepackage{tcolorbox}

\theoremstyle{plain}
 \newtheorem{thm}{Theorem}[section]
 \newtheorem{prop}[thm]{Proposition}
 \newtheorem{lem}[thm]{Lemma}
 \newtheorem{cor}[thm]{Corollary}
\theoremstyle{definition}
 \newtheorem{exm}{Example}[section]
 \newtheorem{dfn}{Definition}[section]
 \newtheorem{nota}{Notation}[section]
\theoremstyle{remark}
 \newtheorem{rem}{Remark}[section]
 \numberwithin{equation}{section}

\renewcommand{\geq}{\geqslant}

\usepackage{hyperref} 
\hypersetup{colorlinks=true,
    linkcolor=blue,         
    citecolor=green,        
    filecolor=magenta,      
    urlcolor=cyan           
}

\title[An Asymptotic Expansion for the Number of 2-Connected Chord Diagrams]{An Asymptotic Expansion for the Number of 2-Connected Chord Diagrams}

\author[Ali Assem Mahmoud]{Ali Assem Mahmoud}
\thanks{}

\address{Mathematics Department, Faculty of Science, Cairo University, Egypt}

\email{aassem@sci.cu.edu.eg}

\address{Department of Combinatorics and Optimization, University of Waterloo, ON, Canada}
\email{ali.mahmoud@uwaterloo.ca}

\begin{document}

 \vspace{18mm} \setcounter{page}{1}\thispagestyle{empty}

\maketitle

\begin{abstract}We derive a functional relation between the generating functions of connected chord diagrams and 2-connected chord diagrams. This relation enables us to calculate an asymptotic expansion for the number of 2-connected chord diagrams on $n$ chords. The asymptotic information obtained from this expansion refines the last established results and provides a simple alternative for calculating the asymptotic behaviour of certain Green functions in Quenched QED and Yukawa theory in the context of quantum field theory.
\end{abstract}

\section{Introduction}\label{introd}

$k$-Connected chord diagrams (also known in the literature as $k$-irreducible linked diagrams) have been studied in many research areas including combinatorics, quantum field theory and bioinformatics \cite{bioinfo, karenbook, kleit}, and, in particular, this research is motivated by the related applications in quenched quantum electrodynamics (QQED) in \cite{michiq} and which the author develops here and in \cite{alipaperphysics}. Namely, the integer sequence $1, 1, 7, 63, 729, 10113, \ldots$  appeared in \cite{broadhurst} and \cite{michiq} as the coefficients of some renormalization counterterms in QQED. Unlike \cite{michiq}, the asymptotic analysis presented here is used to calculate the asymptotic behaviour of these counterterms without the need of singularity analysis, and depending only on the combinatorial interpretation. 

In this paper we shall study the asymptotic behaviour of the number of $2$-connected chord diagrams. Informally speaking, these are chord diagrams which require the removal of at least two chords to get them disconnected. Here we obtain an asymptotic expansion for $(C_{\geq2})_n$, the number of $2$-connected chord diagrams on $n$ chords.  As we have mentioned earlier, in \cite{steinandeveret}, it is shown that the proportion of connected chord diagrams approaches $e^{-1}$ as the number of chords goes to infinity. The work of Stein and Everett in \cite{steinandeveret} addresses a special case of a more general result by Kleitman in \cite{kleit}, where the argument was less detailed. Kleitman argues that the proportion of $k$-connected chord diagrams goes to $e^{-k}$. In \cite{michi}, M. Borinsky showed that the asymptotic behaviour of connected chord diagrams is approximated by a series expansion, in which the first term corresponds to the $e^{-1}$ obtained by Stein and Everett, and earlier by Kleitman; whereas the infinitely many extra terms provide higher precision as needed. Our result here will extend Kleitman's result in very much the same way, this time for the case of $2$-connected chord diagrams. Namely, we obtain an asymptotic expansion for $2$-connected chord diagrams, in which the first term corresponds to the $e^{-2}$ in Kleitman's argument. However, to be able to extract such information about this class of chord diagrams, we will need to work on producing a recursion that relates $2$-connected chord diagrams with connected chord diagrams.

\section{$k$-Connected Chord Diagrams}

\begin{dfn}[Chord diagrams]
A \textit{chord diagram} on $n$ chords (i.e. of size $n$) is geometrically perceived as a circle with $2n$ nodes that are matched into disjoint pairs, with each pair corresponding to a \textit{chord}. 
\end{dfn}

\begin{dfn}[Rooted chord diagrams]
A \textit{rooted} chord diagram is a chord diagram with a selected node. The selected node is called the \textit{root vertex}, and the chord with the root vertex is called the \textit{root chord}. In other words, a rooted chord diagram of size $n$ is a matching of the set $\{1,\ldots,2n\}$. For an algebraic definition, this is the same as a fixed-point free involution in $S_{2n}$. Then the generating series for rooted chord diagrams is 

\begin{equation}\label{rootedchorddiagsgen}
    D(x):=\sum_{n=0}(2n-1)!!\; x^n
\end{equation}

All chord diagrams considered here are going to be rooted and so, when we say a chord diagram we tacitly mean a rooted one.
\end{dfn}

Now, a rooted chord diagram can be represented in a linear order, by numbering the nodes in counterclockwise order, starting from the root which receives the label `$1$'. A chord in the diagram may be referred to as $c=\{a<b\}$, where $a$ and $b$ are the nodes in the linear order.

\begin{dfn}[Intervals]
In the linear representation of a rooted chord diagram, an $interval$ is the space to the right of one of the nodes in the linear representation. Thus, a rooted diagram on $n$ chords has $2n$ intervals.
\end{dfn}
 
For example, this includes the space to the right of the last node (in the linear order). 

\begin{figure}[!htb]
    \centering
    \includegraphics[scale=0.58]{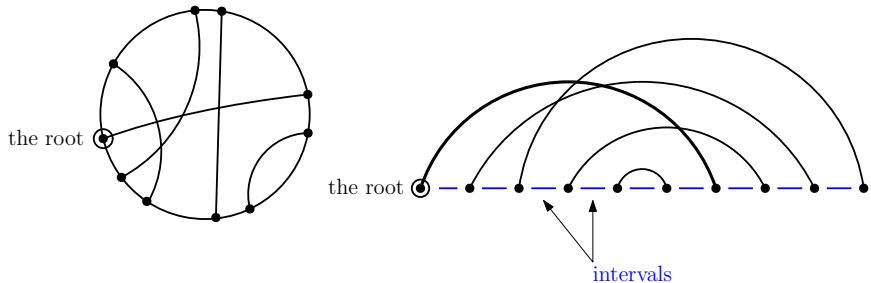}
    \caption{A rooted chord diagram and its linear representation}
\end{figure}

As may be expected by now, the crossings in a chord diagram encode much of the structure and so we ought to give proper notation for them. Namely, in the linear order, two chords $c_1=\{v_1<v_2\}$ and $c_2=\{w_1<w_2\}$ are said to \textit{cross} if $v_1<w_1<v_2<w_2$ or $w_1<v_1<w_2<v_2$.
Tracing all the crossings in the diagram leads to the following definition:

\begin{dfn}[The Intersection Graph] 
Given a (rooted) chord diagram $D$ on $n$ chords, consider the following graph $\mathcal{G}_D$: the chords of the diagram will serve as vertices for the new graph, and there is an edge between the two vertices $c_1=\{v_1<v_2\}$ and $c_2=\{w_1<w_2\}$ if $v_1<w_1<v_2<w_2$ or $w_1<v_1<w_2<v_2$, i.e. if the chords \textit{cross} each other. The graph so constructed is called the \textit{intersection graph} of the given chord diagram.  
\end{dfn}

\begin{rem}
A labelling for the intersection graph can be obtained as follows: give the label $1$ to the root chord;  order the components obtained if the root is removed according to the order of the first vertex of each of them in the linear representation, say the components are $C_1,\ldots,C_n$; and then recursively label each of the components. It is easily verified that a rooted chord diagram can be uniquely recovered from its labelled intersection graph.\\
\end{rem}

\begin{dfn}[Connected Chord Diagrams]\label{c}
A (rooted) chord diagram is said to be \textit{connected} if its  intersection graph is connected (in the graph-theoretic sense). A \textit{connected component} of a diagram is a subset of chords which itself forms a connected chord diagram. The term \textit{root component} will refer to the connected component containing the root chord.
\end{dfn}

\begin{exm}
The diagram $D$ below is a connected chord diagram in linear representation, where the root node is drawn in black.

\begin{center}
\includegraphics[scale=0.5]{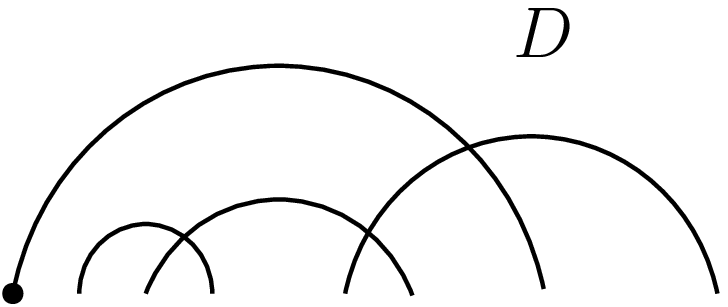}\end{center}
\end{exm}

The generating function for connected chord diagrams (in the number of chords) is denoted by $C(x)$. Thus $C(x)=\sum_{n=0}C_n x^n$, where $C_n$ is the number of connected chord diagrams on $n$ chords. The first terms of $C(x)$ are found to be 
\[C(x)=x+x^2+4x^3+27\;x^4+248\;x^5+\cdots\;;\]
the reader may refer to OEIS sequence \href{https://oeis.org/A000699}{A000699}  for more coefficients. 
The next lemma lists some classic decompositions for chord diagrams (see \cite{flajoletchords} for example).

\begin{lem}\label{cd}
  If $D(x), C(x)$ are the generating series for chord diagrams and connected chord diagrams respectively, then 
  \begin{enumerate}
      \item[$\mathrm{(i)}$]  $D(x)=1+C(xD(x)^2)$, \label{i1}
      \item[$\mathrm{(ii)}$] $D(x)=1+xD(x)+2x^2D'(x)$, and \label{i2}
      \item[$\mathrm{(iii)}$] $2xC(x)C'(x)=C(x)(1+C(x))-x$.\label{i3}
  \end{enumerate}

  \end{lem}
  
  \begin{proof} We sketch the underlying decompositions as follows: 
  \begin{enumerate}
    \item[$\mathrm{(i)}$] The `one' term is for the empty chord diagram. Now, given a nonempty chord diagram, we see that for every chord in the root component there live two chord diagrams to the right of its two ends. This gives the desired decomposition.
     
     \item[$\mathrm{(ii)}$] There are three situations for a root chord: it is either non-existent (empty diagram); or it is concatenated with a following diagram; or the root chord has its right end landing in one of the intervals of a diagram. These situations correspond respectively with the terms in (ii).
     
   \item[$\mathrm{(iii)}$] Can be derived from (i) and (ii). Nevertheless, it can be also shown as follows: if we remove the root chord what is left is a sequence of connected components, with each component having a special interval (through wich the root used to pass) which cannot be the last interval (see the figure below). Thus each of these components is counted according to the generating function $2xC'(x)-C(x)$.
   \begin{center}
   \includegraphics[scale=0.65]{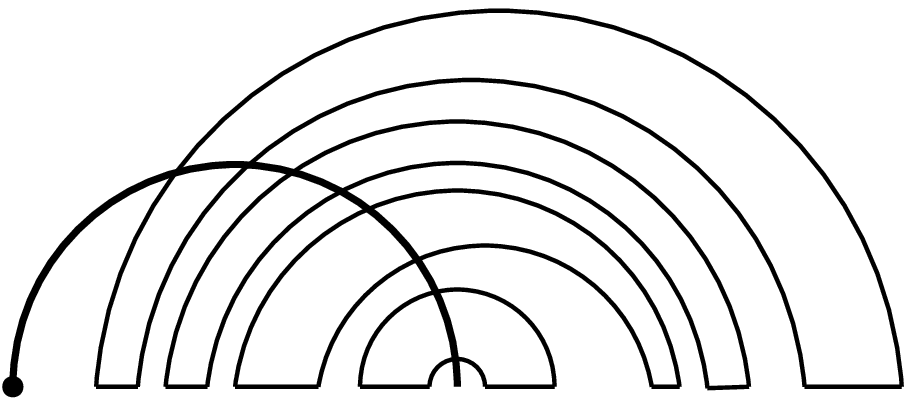}   
   \end{center}
   
   This decomposition gives that
   \[C(x)=\displaystyle\frac{x}{1-(2xC'(x)-C(x))},\]
   and the result follows.
     \end{enumerate}
  \end{proof}

The main object we use throughout is chord diagrams with certain degrees (strengths) of connectivity.

\begin{dfn}[$k$-Connected Chord Diagrams]
A chord diagram on $n$ chords is said to be $k$-\textit{connected} if there is no set $S$ of consecutive endpoints, with $|S|< 2n-k$, $S$ is paired with less than $k$ endpoints not in $S$ (here we assume the endpoints are consecutive in the sense of the linear representation). In other words, the diagram requires the deletion of at least $k$ chords to become disconnected. A $k$-\textit{connected} diagram which is not $k+1$-\textit{connected} will be said to have \textit{connectivity} $k$.
\end{dfn}

\begin{exm} The diagram in Figure \ref{3not4} is 3-connected since it can not be disconnected with the removal of fewer than 3 chords, but it is not 4-connected.

\begin{figure}
    \centering
    \includegraphics[scale=0.65]{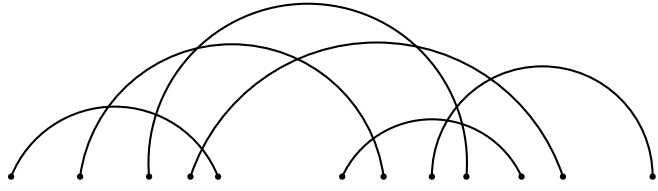}
    \caption{A diagram that is 3-connected but not 4-connected}
    \label{3not4}
\end{figure}
\end{exm}

\begin{dfn}[Cuts and Reasons for Connectivity-$k$]\label{connectv1dfn}
Given a connectivity-$k$ diagram, a set of size $k$ of chords is called a $cut$ if its removal disconnects the diagram. Equivalently, a set $T$ of $k$ chords in a connectivity-$k$ diagram is a cut if there exists a sequence $S$ of consecutive end points such that $|S|< 2n-k$ and all the end points in $S$ are paired together except for $k$ endpoints from the $k$ chords in $T$. Such a sequence $S$ will be called a \textit{reason for connectivity-$k$}.  See Figure \ref{reason} below for illustration.
\end{dfn}

\begin{figure}[h]
    \centering
    \includegraphics[scale=0.75]{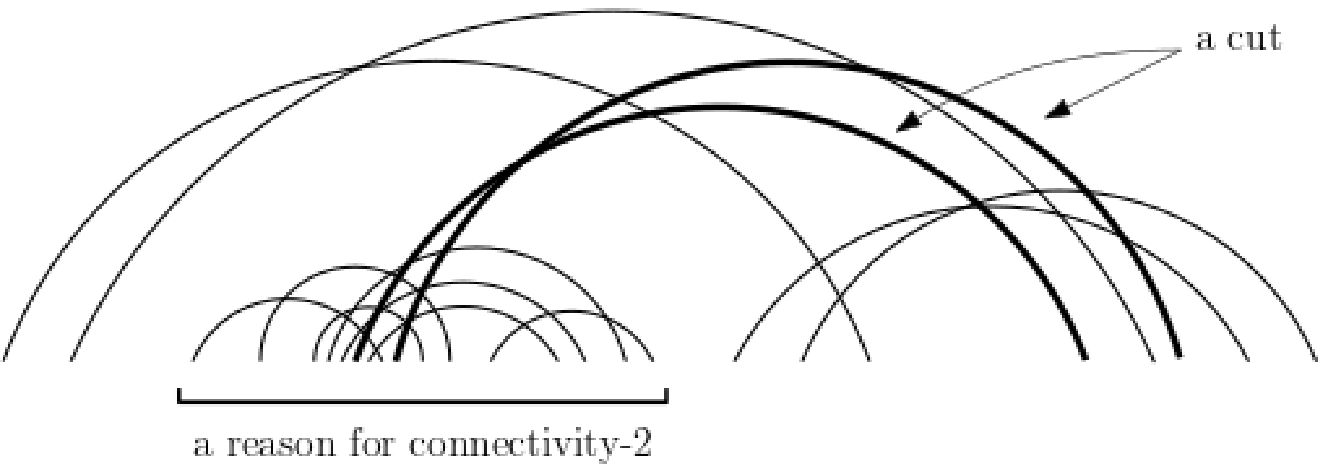}
    \caption{}
    \label{reason}
\end{figure}

\begin{nota}

 For the generating functions we shall use the following notation: $C_{\geq k}(x)$ (or $C_{\geq k}$) will denote $k$-connected diagrams whereas $C_k(x)$ (or $C_k$) denotes diagrams with connectivity $k$. So for example $C(x)=C_1(x)+C_{\geq 2}(x)$. 
 \end{nota}

A computation of the first coefficients gives 
\begin{equation}
    \begin{split}
        C(x)&=x+x^2+4x^3+27x^4+248x^5+2830x^6+\cdots\\
        C_1(x)&=x+3x^3+20x^4+185x^5+2101x^6+\cdots\\
        C_{\geq2}(x)&=x^2+x^3+7 x^4+63 x^5+729 x^6+\cdots
    \end{split}
\end{equation}

\section{Functional Recurrence for $2$-Connected Diagrams}\label{functionalrecurrence2connected}

In \cite{michi}, the (classic) functional relation $D(x)=1+C(xD^2)$ provided the suitable grounds for deriving information about the asymptotic behaviour of $C_n$, the number of connected chord diagrams on $n$ chords. The composition of maps in the second term in this relation transforms nicely into a product when taking the alien derivative $\mathcal{A}_{1/2}^2$ (see Appendix \ref{factorially} for definitions). In the aftermath of our meeting in the Canadian Mathematical Society session about chord diagrams (Dec. 2019), M. Borinsky suggested to the author that it may be possible to obtain similar functional relations for the higher connectivity diagrams. This was motivated by the asymptotic pattern shown in Kleitman's results \cite{kleit}. In this section we derive such a functional relation for $2$-connected chord diagrams, and will use it later to study the asymptotic behaviour of the number of $2$-connected chord diagrams. However, just as the case for general graphs, it is not clear whether $3$-connected diagrams and $k$- connected diagrams in general do follow similar relations. 
\begin{prop}\label{myproposition2connected}
The following functional relation between connected and $2$-connected diagrams holds:

\begin{equation}
    C=\displaystyle\frac{C^2}{x}-C_{\geq2}\left(\displaystyle\frac{C^2}{x}\right). \label{c2con}
\end{equation}
\end{prop}

\begin{proof}
Assume that a connected chord diagram $\mathbf{C}$ is given. We can determine the maximal sequences of consecutive end points that are reasons for connectivity-$1$. A sequence $S=s_1s_2\ldots s_m$ of consecutive end points is of this type if and only if \begin{enumerate}
    \item $S$ is a reason of connectivity-$1$ corresponding to a cut chord $c$ that has exactly one end point inside $S$, say this is $s_i\;$ where $1<i < m$.
    \item $S$ is not contained in any other reason for connectivity-$1$.

\end{enumerate}

However, these reasons for connectivity-$1$ may overlap (see Figure \ref{reasonprob}), and so we will need to devise a canonical way for partitioning our diagram in terms of these maximal sequences.  

\begin{figure}[h]
   \center
    \includegraphics[scale=0.74]{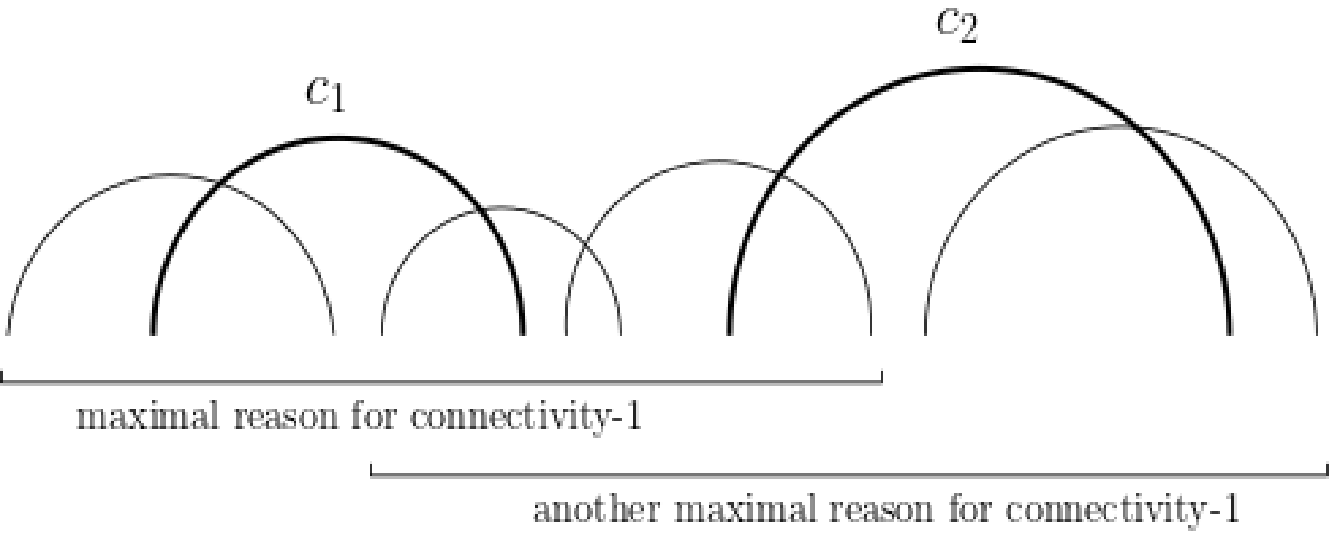}
    \caption{}
    \label{reasonprob}
\end{figure}

\textbf{Case 1:} $\mathbf{C}$ is the \textbf{single chord diagram}. In this case we do nothing, and the contribution to the generating function is just $x$.

In the next cases we generally assume $\mathbf{C}$ is not the single chord diagram.

\textbf{Case 2:} The root endpoint $r_0$ (left endpoint of the root chord of $\mathbf{C}$) is \textbf{not} contained in any reason for connectivity-$1$. In this case we determine the maximal reasons for connectivity-$1$ that are obtained through the next procedure by moving from left to right. Such a diagram generally looks like the example in  Figure \ref{looklike} below.

\begin{figure}[h]
   \center
    \includegraphics[scale=0.4]{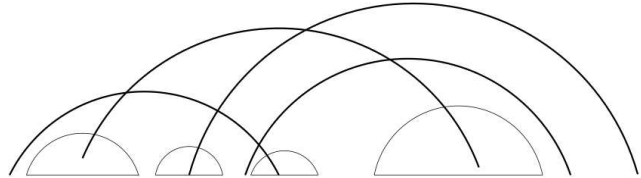}
    \caption{An illustration for Case 2}
    \label{looklike}
\end{figure}

Consider the diagram $\mathbf{C^\times}$ obtained from $\mathbf{C}$ as follows: 
\begin{enumerate}
    \item Starting from the left, determine the first endpoint that is included in some reason for connectivity-$1$. Let's denote it temporarily by $s_1$. Move to step 5 if the diagram is $2$-connected and no such endpoint exists. 
    
    \item Determine the maximal reason $S=s_1s_2\ldots s_m$ for connectivity-$1$ that contains $s_1$ by consecutively trying to include the next endpoints to the right. Assume $S$ corresponds to a cut chord $c$ that has the end point $s_i$, say.
    
    \item Let $\mathbf{C^\times}$ be the diagram obtained by removing the sub-diagram induced by $S-{s_i}$, i.e. we remove $S$ without removing $c$. 
    \item Update by setting $\mathbf{C}=\mathbf{C^\times}$, and go back to step 1.\\
    
    \item Output $\mathbf{C^\times}$.

\end{enumerate}

\textbf{Observation 1:} 
Notice that in the process of extracting $\mathbf{C^\times}$ the diagram remains connected, this is because any of the removed sub-diagrams has been connected to the rest of $\mathbf{C}$ through a single cut chord which is not removed.

Clearly, $\mathbf{C^\times}$ will not preserve any original reason for connectivity-1 in $\mathbf{C}$. Moreover, notice that again since each sub-diagram removed has only been connected to the rest of $\mathbf{C}$ through a single cut chord (which is kept), the process should not affect the connectivity of the rest of $\mathbf{C}$ neither will create new cuts.

\textbf{Observation 2:} 
Also, step 1 is exclusive throughout the procedure. Indeed, if there is no such endpoint in a connected diagram (Observation 1) then the diagram is either $2$-connected or is the single chord diagram (it can't be empty). The latter however never occurs: Initially the diagram is not the single chord diagram by our assumption. Further, $\mathbf{C}$ is not reduced to a single chord diagram at any iteration since this should imply that the root endpoint $r_0$ is contained in a reason for connectivity-1. Therefore the procedure eventually halts and the output is $2$-connected. 

\textbf{Observation 3:} It is important to note that also the last endpoint in $\mathbf{C}$ is not included in any reason for connectivity-1, for this will imply the same for $r_0$.
   
To summarize the procedure above, we are removing maximal reasons of connectivity-$1$ that appear in a certain order when moving from left to right, without removing their corresponding cuts. This is illustrated in Figures  \ref{con11} and \ref{con1removed}.

This gives a reversible decomposition into a $2$-connected where each endpoint, except the first and last endpoints, is assigned to a connected chord diagram counted by one less chord. In other words, we will count each middle chord (i.e. whose endpoints are not the root nor the last endpoint) in $\mathbf{C}^\times$ when counting the connected diagram for its right endpoint by keeping it as a root for this diagram, while on the other hand, the diagram for the left endpoint will be counted by one less chord to avoid overcounting.

This can also be viewed as follows:

Given a connected chord diagram $\mathbf{C}$ (which is not the single chord) we undergo the described procedure to get 

\begin{enumerate}
    \item a $2$-connected chord diagram $\mathbf{C}^\times$,
    \item the root chord $c_{r}$ corresponds to a connected chord diagram that consists of $c_r$ and the diagram attached to the right endpoint of the root, in which we will keep the root. 
    \item the chord $c_l$ carrying the last endpoint of $\mathbf{C}$ corresponds to a connected chord diagram that consists of $c_l$ kept as a root for whatever the diagram attached to the left endpoint.
    
    \item Every middle chord $c$ can be replaced with a pair of diagrams corresponding to right and left endpoints. The diagram for the left endpoint has its root a copy of $c$ that is not going to be counted and is connected; while the diagram for the right endpoint keeps $c$ and is connected as well.
\end{enumerate}

In terms of generating functions  the contribution of Case 2 is seen now to be:

\begin{equation}
    C(x)^2 \;\left[\left. \displaystyle\frac{C_{\geq2}(t)}{t^2}\right|_{t=C(x)^2/x}\right],\label{case2}
\end{equation}
where we divide by $t^2$ to account for the fact that two of the chords are treated differently (namely $c_r$ and $c_l$). Each of these two chords contributes with $C(x)$ as shown above.

\begin{figure}[h]
   \center
    \includegraphics[scale=0.67]{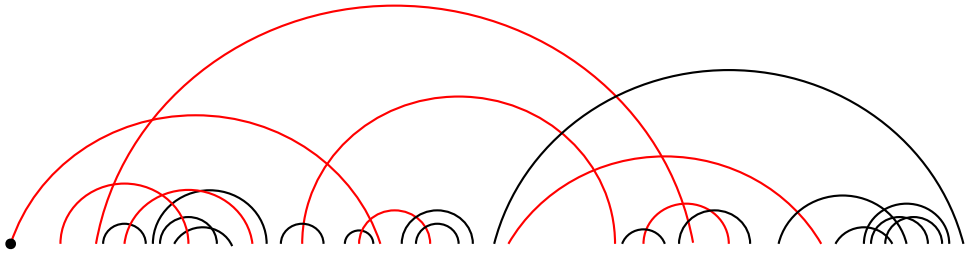}
    \caption{A diagram $\mathbf{C}$ with cuts highlighted (red).}
    \label{con11}
\end{figure}

\begin{figure}[h]
   \center
    \includegraphics[scale=0.5]{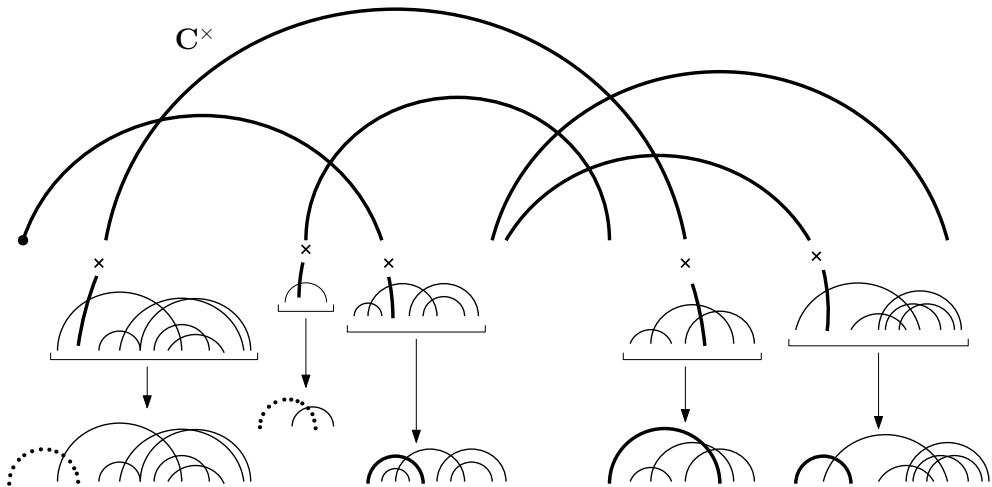}
    \caption{Maximal reasons obtained through the procedure (underlined), $\mathbf{C^\times}$ (bold).}
    \label{con1removed}
\end{figure}

\textbf{Case 3:} The root endpoint $r_0$ (left endpoint of the root chord of $\mathbf{C}$) is \textbf{contained} in a reason for connectivity-$1$. In this case we determine the maximal reason for connectivity-1 containing $r_0$, donted by $S$, by consecutively checking every endpoint to the right of $r_0$. Let $c^*$ be the corresponding cut for $S$. Now, by the maximality of $S$ it must be that none of the reasons for connectivity-1 that come later could be extended to contain $S$.  This means that the diagram obtained by removing $S$ (without removing $c^*$) is of the type considered in Case 2 above. The diagram will generally be structured as  in Figure \ref{struc}. Then the contribution to the generating function is

\begin{equation}
   \displaystyle\frac{C(x)-x}{x}\;. \;C(x)^2 \;\left[\left. \displaystyle\frac{C_{\geq2}(t)}{t^2}\right|_{t=C(x)^2/x}\right], \label{case3}
\end{equation}

where the factor of $\displaystyle\frac{C(x)-x}{x}$ corresponds to the sub-diagram induced by $S$ together with $c_r$: the $(C-x)$ since $S$ is always nonempty in this case,  and we divide by $x$ since $c_r$ is counted with the rest of the diagram. 

 \begin{figure}[h]
   \center
    \includegraphics[scale=0.67]{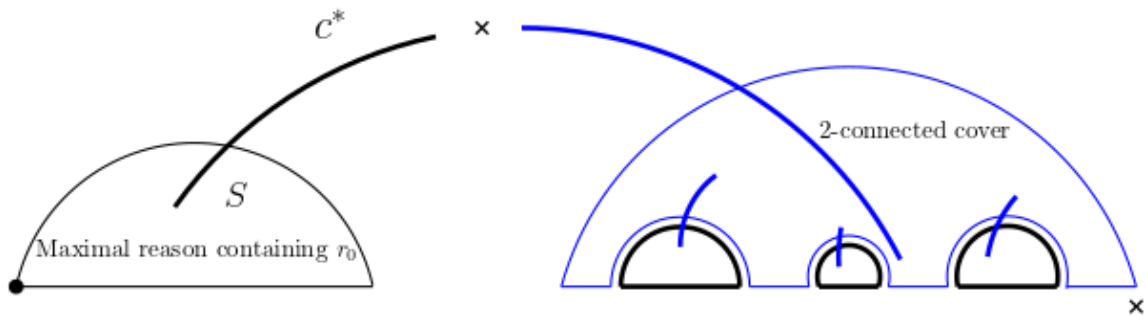}
    \caption{$S$ is the maximal reason containing $r_0$, and so the rest of the diagram should be covered with a $2$-connected sub-diagram (blue).}
    \label{struc}
\end{figure}

Thus, by combining the findings of the three cases we have 

\begin{align*}
   C(x)&=x+C(x)^2 \;\left[\left. \displaystyle\frac{C_{\geq2}(t)}{t^2}\right|_{t=C(x)^2/x}\right]+
   \displaystyle\frac{C(x)-x}{x}\;. \;C(x)^2 \;\left[\left. \displaystyle\frac{C_{\geq2}(t)}{t^2}\right|_{t=C(x)^2/x}\right]\\
   &\\
  &=x+\displaystyle\frac{C(x)^3}{x}\;.  \;\left[\left. \displaystyle\frac{C_{\geq2}(t)}{t^2}\right|_{t=C(x)^2/x}\right]\\
  &\\
  &=x+\displaystyle\frac{x}{C(x)}\;.\; \left(\displaystyle\frac{C(x)^2}{x}\right)^2\;. \;\left[\left. \displaystyle\frac{C_{\geq2}(t)}{t^2}\right|_{t=C(x)^2/x}\right]\\
  &\\
  &=x+\displaystyle\frac{x}{C(x)}\;.\;C_{\geq2}\left(\displaystyle\frac{C(x)^2}{x}\right),
\end{align*}

and the result now follows.

\end{proof}

For future reference, we include the first terms of the expressions involved in the previous decomposition. The reader can check that the sum of $x$ plus  lines 3 and 4 in the next table gives the first terms of $C(x)$.

\begin{table}[h]
\center
\begin{tabular}{|c|c||c|c|c|c|c|c|c|}\hline
&             &$x^0$ & $x$   & $x^2$  & $x^3$ & $x^4$ & $x^5$ & $x^6$ \\\hline\hline
1&$C^2/x$     &   0  & 1     & 2      & 9     & 62    & 566   & 6372      \\\hline
2&$\left. \displaystyle\frac{C_{\geq2}(t)}{t^2}\right|_{t=C(x)^2/x}$      
              &   1  & 1     & 9      & 100   & 1323  & 20088 & 342430         \\\hline
3&$C(x)^2 \;\big[\left. \displaystyle\frac{C_{\geq2}(t)}{t^2}\right|_{t=C(x)^2/x}\big]$      
              &   0  & 0     & 1      & 3     & 20    & 189   & 2232      \\\hline
4&$\displaystyle\frac{C(x)-x}{x}\;. \;C(x)^2 \;\big[\left. \displaystyle\frac{C_{\geq2}(t)}{t^2}\right|_{t=C(x)^2/x}\big]$ 
              &   0  & 0     & 0      & 1     & 7     & 59    & 598  
                  \\\hline

 \end{tabular}\vspace{0.3cm}\caption{The first coefficients of the series involved in the terms of the decomposition of $C_{\geq2}$. }
\label{table2}
\end{table}

\section{Asymptotics of the number of 2-connected chord diagrams}

In this section we will see how to successfully estimate the number $(C_{\geq2})_n$ of $2$-connected diagrams when $n$ is large. The asymptotic behaviour obtained here will extend Kleitman's result \cite{kleit} and will shed light on an unexplained pattern for the images of the alien derivative. It turns out that $\mathcal{A}_{\frac{1}{2}}^2C_{\geq2}$ takes the form of a rational function in $C_{\geq2}$ times the exponential of a quadratic expression in the reciprocal of that rational function. This was exactly the same case for $\mathcal{A}_{\frac{1}{2}}^2C$ (as well as monolithic diagrams and simple permutations). We will proceed now by applying a suitable alien derivative as was done before for connected chord diagrams.

In the previous section we have seen that \begin{equation*}
    C=\displaystyle\frac{C^2}{x}-C_{\geq2}\left(\displaystyle\frac{C^2}{x}\right). 
\end{equation*}

We will start by applying the alien derivative $\mathcal{A}_{\frac{3}{2}}^2$, which is allowed since $C(x)\in\mathbb{R}[[x]]_{\frac{1}{2}}^2\subset \mathbb{R}[[x]]_{\frac{3}{2}}^2 $ by Corollary \ref{corplusm1}.

\begin{align*}
    \big(\mathcal{A}_{\frac{3}{2}}^2C\big)(x)
    &=\mathcal{A}_{\frac{3}{2}}^2\left(\displaystyle\frac{C(x)^2}{x}\right)-\mathcal{A}_{\frac{3}{2}}^2\left(C_{\geq2}\left(\displaystyle\frac{C(x)^2}{x}\right)\right)(x)\\
    &=\mathcal{A}_{\frac{1}{2}}^2\big(C(x)^2\big)-\mathcal{A}_{\frac{3}{2}}^2\left(C_{\geq2}\left(\displaystyle\frac{C(x)^2}{x}\right)\right)(x)\\
    &=2 C \big(\mathcal{A}_{\frac{1}{2}}^2C\big)(x)-\mathcal{A}_{\frac{3}{2}}^2\left(C_{\geq2}\left(\displaystyle\frac{C(x)^2}{x}\right)\right)(x),
\end{align*}

by the linearity of $\mathcal{A}_{\frac{3}{2}}^2$ and by Proposition \ref{plusm2}. Rearrange and use  Proposition \ref{plusm1} to get

\begin{align*}
   (2C-x)\big(\mathcal{A}_{\frac{1}{2}}^2C\big)(x)= \mathcal{A}_{\frac{3}{2}}^2\left(C_{\geq2}\left(\displaystyle\frac{C(x)^2}{x}\right)\right)(x).
\end{align*}

Now to get rid of the decomposition on the right we appeal to Theorem \ref{chaintheorem}:

\begin{align*}
   (2C-x)\big(\mathcal{A}_{\frac{1}{2}}^2C\big)(x) 
   &= C'_{\geq2}\bigg(\displaystyle\frac{C^2}{x}\bigg)  \mathcal{A}_{\frac{3}{2}}^2\bigg(\displaystyle\frac{C^2}{x}\bigg)
   +\bigg(\displaystyle\frac{x^2}{C^2}\bigg)^\frac{3}{2} e^{\frac{C^2/x-x}{2x C^2/x}} \left(\mathcal{A}_{\frac{3}{2}}^2C_{\geq2}\right)\left(\displaystyle\frac{C^2}{x}\right)\\
   &\\
   &=2C\big(\mathcal{A}_{\frac{1}{2}}^2C\big)(x)\; C'_{\geq2}\bigg(\displaystyle\frac{C^2}{x}\bigg)+\displaystyle\frac{x^3}{C^3}\; e^{\frac{C^2/x-x}{2x C^2/x}} \left(\mathcal{A}_{\frac{3}{2}}^2C_{\geq2}\right)\left(\displaystyle\frac{C^2}{x}\right).
\end{align*}

Equation (\ref{c2con}) and Lemma \ref{cd} give that 

\begin{align*}
C'&=\displaystyle\frac{2xCC'-C^2}{x^2}\left[1- C'_{\geq2}\bigg(\displaystyle\frac{C^2}{x}\bigg)\right]\\
&\\
&=\displaystyle\frac{C^2+C-x-C^2}{x^2}\left[1- C'_{\geq2}\bigg(\displaystyle\frac{C^2}{x}\bigg)\right]\\
&\\
&=\displaystyle\frac{C-x}{x^2}\left[1- C'_{\geq2}\bigg(\displaystyle\frac{C^2}{x}\bigg)\right].
\end{align*}

Substituting into our equation we get 

\begin{align*}
   \left(\displaystyle\frac{2x^2CC'}{C-x}-x\right)\big(\mathcal{A}_{\frac{1}{2}}^2C\big)(x)
   &=\displaystyle\frac{x^3}{C^3}\; e^{\frac{C^2/x-x}{2x C^2/x}} \left(\mathcal{A}_{\frac{3}{2}}^2C_{\geq2}\right)\left(\displaystyle\frac{C^2}{x}\right).
\end{align*}
   
Now, by \cite{michi1},  $\big(\mathcal{A}_{\frac{1}{2}}^2C\big)(x)=\displaystyle\frac{1}{\sqrt{2\pi}}\frac{x}{C(x)}\;e^{\frac{-1}{2x}(C^2+2C)}$, and 
hence 
\begin{align*}
    \left(\mathcal{A}_{\frac{3}{2}}^2C_{\geq2}\right)\left(\displaystyle\frac{C^2}{x}\right)
   &=\displaystyle\frac{C^3}{x^3}\left(\displaystyle\frac{2x^2CC'}{C-x}-x\right)\big(\mathcal{A}_{\frac{1}{2}}^2C\big)(x)\; e^{\frac{x-C^2/x}{2x C^2/x}}\\
   &\\
   &=\displaystyle\frac{C^3}{x^3}\cdot\displaystyle\frac{x(C^2+C-x)-xC+x^2}{C-x}\cdot\big(\mathcal{A}_{\frac{1}{2}}^2C\big)(x)\;\cdot e^{\frac{x-C^2/x}{2x C^2/x}}\\
   &\\
   &=\displaystyle\frac{C^3}{x^3}\cdot\displaystyle\frac{xC^2}{C-x}\cdot\displaystyle\frac{1}{\sqrt{2\pi}}\frac{x}{C(x)}\;e^{\frac{-1}{2x}(C^2+2C)}\;\cdot e^{\frac{x-C^2/x}{2x C^2/x}}.
\end{align*}

Since the $LHS$ is a function in $\displaystyle\frac{C^2}{x}$, applying Proposition \ref{plusm1} gives that 

\[\left(\mathcal{A}_{\frac{3}{2}}^2C_{\geq2}\right)\left(\displaystyle\frac{C^2}{x}\right)=\displaystyle\frac{C^2}{x}\cdot \left(\mathcal{A}_{\frac{1}{2}}^2C_{\geq2}\right)\left(\displaystyle\frac{C^2}{x}\right) .\]

Back to our equation, we thus have 
\begin{equation}\label{pre}
    \left(\mathcal{A}_{\frac{1}{2}}^2C_{\geq2}\right)\left(\displaystyle\frac{C(x)^2}{x}\right)
   =\displaystyle\frac{1}{\sqrt{2\pi}}\cdot\displaystyle\frac{C^2}{C-x}\cdot e^{\frac{-1}{2x}(C^2+2C)}\;\cdot e^{\frac{x-C^2/x}{2x C^2/x}}
   =\displaystyle\frac{1}{\sqrt{2\pi}}\cdot\displaystyle\frac{C^2}{C-x}\cdot e^{\frac{-1}{2x}[C^2+2C+1-\frac{x^2}{ C^2}]}.
\end{equation}

Since the power series $\displaystyle\frac{C(x)^2}{x}$ is invertible, we let $y(x)$ be such that $\displaystyle\frac{C(y)^2}{y}=x$. In that case equation (\ref{c2con}) gives
\[C(y)=x-C_{\geq2}(x),\]

and hence \[y(x)=\displaystyle\frac{(x-C_{\geq2}(x))^2}{x}.\]
Substituting $y(x)$ for $x$ in equation (\ref{pre}) we get

\begin{align*}\label{done}
    \left(\mathcal{A}_{\frac{1}{2}}^2C_{\geq2}\right)(x)
   &=\displaystyle\frac{1}{\sqrt{2\pi}}\cdot\displaystyle\frac{(x-C_{\geq2})^2}
   {\left((x-C_{\geq2})-\displaystyle\frac{(x-C_{\geq2})^2}{x}\right)}\cdot  e^{\frac{-1}{2}\left[x+\frac{2x}{(x-C_{\geq2})}+\frac{x}{(x-C_{\geq2})^2}-\frac{1}{x}\right]}\\
   &\\
   &=\displaystyle\frac{1}{\sqrt{2\pi}}\cdot\displaystyle\frac{x\;(x-C_{\geq2})}
   {\left(x-x+C_{\geq2}\right)}\cdot  e^{\frac{-1}{2x}\left[x^2+\frac{2x}{(1-C_{\geq2}/x)}+\frac{1}{(1-C_{\geq2}/x)^2}-1\right]}\\
   &\\
   &=\displaystyle\frac{1}{\sqrt{2\pi}}\cdot\displaystyle\frac{x^2}
   {\left(\displaystyle\frac{C_{\geq2}}{(1-C_{\geq2}/x)}\right)}\cdot  e^{\frac{-1}{2x}\left[\left(\frac{1}{(1-C_{\geq2}/x)}+x\right)^2-1\right]}.
\end{align*}

In other words,

\begin{equation}\label{alien2con}
    \left(\mathcal{A}_{\frac{1}{2}}^2C_{\geq2}\right)(x)=
    \displaystyle\frac{1}{\sqrt{2\pi}}\cdot\displaystyle\frac{x^2}
   {C_{\geq2}S}\cdot  e^{\frac{-1}{2x}\left[\left(S+x\right)^2-1\right]},
\end{equation}

where $S(x)=\displaystyle\frac{1}{\Big(1-\displaystyle\frac{C_{\geq2}(x)}{x}\Big)}$ is the generating series for sequences of $2$-connected chord diagrams counted by one less chord.

Finally it is noteworthy to see that the image $\mathcal{A}_{\frac{1}{2}}^2C_{\geq2}$ of $C_{\geq2}(x)$ under the alien derivative is of the form of a rational function of $C_{\geq2}$ times the exponential of a quadratic expression in the rational function. The same pattern has been observed in the case of connected chord diagrams. From another point of view, one can see that $xS(x)$ also counts connectivity-1 diagrams in which only the root chord is a cut. 

The evaluation of $\mathcal{A}_{\frac{1}{2}}^2C_{\geq2}$ will enable us to derive information about the asymptotic behaviour which strongly extend the result by Kleitman in \cite{kleit}.  First let us list the first few terms of the functions involved.

\begin{table}[h]
\center
\begin{tabular}{|c|c||c|c|c|c|c|c|c|c|}\hline
 &            &$x^0$ & $x$   & $x^2$  & $x^3$ & $x^4$ & $x^5$ & $x^6$ & $x^7$                           \\\hline\hline
1&$S(x)=1/{\Big(1-\displaystyle\frac{C_{\geq2}(x)}{x}\Big)}$     
              &   1  & 1     & 2      & 10    & 82    & 898   & 12018 & 
              \\\hline
2&$(S+x)^2$      
              &   1  & 4     & 8      & 28    & 208   & 2164  & 28056 & 
              \\\hline
3&$\displaystyle\frac{1}{2x}\left[(S+x)^2-1\right]$      
              &   2  & 4     & 14     & 104   & 1082  & 14028 &       & 
              \\\hline
4&${C_{\geq2}\cdot S}$ 
              &   0  & 0     & 1      & 2     & 10    & 82    & 898   & 12018
              \\\hline
5&$\displaystyle\frac{x^2}{C_{\geq2}\cdot S}$
              &   1  & -2    &-6      & -50   & -574  & -8082 &       &  
              \\\hline
6&$e^2\;\cdot\;\exp\big\{\displaystyle\frac{-1}{2x}\left[(S+x)^2-1\right]\big\}$      
              &   1  & -4    & -6     & $\frac{-176}{3}$   & $\frac{-2008}{3}$  & $\frac{-46636}{5}$  &       &   
              \\\hline
 \end{tabular}\vspace{0.3cm}\caption{The first terms of the series involved in calculating $\mathcal{A}^2_{\frac{1}{2}}C_{\geq2}(x)$.  }
\label{table3}
\end{table}

Note  that we are willing to display the factor of $e^{-2}$ that comes from $\;\exp\big\{\displaystyle\frac{-1}{2x}\left[(S+x)^2-1\right]\big\}$ and that is why the last row in  Table \ref{table3} is multiplied by $e^2$.

The computation then gives
\begin{align}\label{asymptoticsC2}
    \left(\mathcal{A}_{\frac{1}{2}}^2C_{\geq2}\right)(x)
    &=\displaystyle\frac{1}{\sqrt{2\pi}}\cdot\displaystyle\frac{x^2}
   {C_{\geq2}S}\cdot  e^{\frac{-1}{2x}\left[\left(S+x\right)^2-1\right]} \nonumber \\
    &\nonumber\\
    &=\displaystyle\frac{e^{-2}}{\sqrt{2\pi}}\big[1-6x-4x^2-\displaystyle\frac{218}{3}x^3-
    890x^4-\displaystyle\frac{196838}{15}x^5-\cdots\big]
\end{align}

Now, by Definition \ref{fdps} of factorially divergent power series and Definition \ref{map} of the alien derivative $\mathcal{A}_{\frac{1}{2}}^{2}$, and since $\Gamma^2_{\frac{1}{2}}(n)=\sqrt{2\pi} (2n-1)!!$, we obtain that, for all $R\in\mathbb{N}_0$, the number $(C_{\geq2})_n$
of $2$-connected diagrams on $n$ chords satisfies
\begin{align}
    (C_{\geq2})_n
    &=\overset{R-1}{\underset{k=0}{\sum}} [x^k] \left(\mathcal{A}_{\frac{1}{2}}^2C_{\geq2}\right)(x) \cdot \Gamma^2_{\frac{1}{2}}(n-k)+\mathcal{O}(\Gamma^2_{\frac{1}{2}}(n-R))\nonumber\\
    &=\sqrt{2\pi}\;\overset{R-1}{\underset{k=0}{\sum}} [x^k] \left(\mathcal{A}_{\frac{1}{2}}^2C_{\geq2}\right)(x) \cdot (2(n-k)-1)!! +\mathcal{O}((2(n-R)-1)!!),\nonumber
\end{align}

and hence the first few terms in this asymptotic expansion are given by 

\begin{align}
    (C_{\geq2})_n 
    & = e^{-2}  \bigg((2n-1)!!-6(2n-3)!!-4(2n-5)!!-
    \displaystyle\frac{218}{3}(2n-7)!!- \nonumber\\
    &\; \qquad       -890(2n-9)!!-\displaystyle\frac{196838}{15}(2n-11)!!-\cdots\bigg)
    \nonumber\\
    &\;        \nonumber\\
    & = e^{-2} (2n-1)!!\bigg(1-\displaystyle\frac{6}{2n-1}-\displaystyle\frac{4}{(2n-3)(2n-1)}-\displaystyle\frac{218}{3}\displaystyle\frac{1}{(2n-5)(2n-3)(2n-1)}- \nonumber\\
    &\;   \qquad     -\displaystyle\frac{890}{(2n-7)(2n-5)(2n-3)(2n-1)}-\displaystyle\frac{196838}{15}\displaystyle\frac{1}{(2n-9)\cdots(2n-1)}-\cdots\bigg).\nonumber\\
    &\label{computC2asympt}
\end{align}

The result by Kleitman \cite{kleit} corresponds to the first term in this expansion. By the above approach, any precision can be achieved and an arbitrary number of terms can be produced. 

Equation \ref{computC2asympt} also shows that a randomly chosen chord diagram on $n$ chords is $2$-connected with a probability of \[\displaystyle\frac{1}{e^2}\Big(1-\displaystyle\frac{3}{n}\Big)+\mathcal{O}(1/{n^2}).\]

In the next section we will see that this expansion also corresponds to the asymptotics of the number of skeleton quenched QED vertex diagrams \cite{michiq}. In that context the first five terms of the above expansion were conjectured  by D. J. Broadhurst on a numerical evidence (see page 38 in \cite{michiq}) in studying zero-dimensional field theory.


\section{Connection with Zero-Dimensional QFT}\label{quenchedsec}

In the next part of the paper we will see that some of the integer sequences produced in studying $2$-connected chord diagrams appear in the context of zero-dimensional quantum field theory. Note that in this situation the partition function transforms into a series of graphs since no actual Feynman integral shall remain. On another  level, the Feynman rules  will be represented as a character from $\mathcal{H}$ to the algebra $\mathbb{R}[[\hbar]]$ \cite{michi}. We managed to establish the relation between $2$-connected chord diagrams and some of the observables in \textit{quenched QED}. In \cite{michiq} the asymptotics for these sequences are obtained through a \textit{singularity analysis} approach. We will be able to get the same asymptotics through an enumerative approach. Factorially divergent power series are, as expected, used in both approaches, and hence we will regularly appeal to theorems from Section \ref{factorially}. First we will briefly set-up the context in perturbation theory. In most parts we follow the notation in \cite{michiq}.

Recall the basic path integral formulation of QFT and notice that for   zero-dimensional QFT the path integral for the partition function becomes an ordinary integral given for example by 

\[Z(\hbar):=\int_\mathbb{R} \displaystyle\frac{1}{\sqrt{2\pi\hbar}} e^{\frac{1}{\hbar}\left(-\frac{x^2}{2a}+V(x)\right)}dx,\]
where $V(x)\in x^3\mathbb{R}[[x]]$ is the \textit{potential} and the exponent $-\frac{x^2}{2a}+V(x)$ is the \textit{action} and denoted by $\mathcal{S}$. As known, this integral generally has singularities, and even as a series expansion it generally have a singularity at zero. In \cite{michiq}, the expansion is treated as a formal power series and the focus is on studying the asymptotics of the coefficients.

Recall that Gaussian integrals satisfy \[\int_\mathbb{R} \displaystyle\frac{1}{\sqrt{2\pi\hbar}} e^{-\frac{x^2}{2a\hbar}} x^{2n}dx=\sqrt{a}(a\hbar)^n (2n-1)!!, \] were only the even powers are considered since the integral vanishes for odd powers. This enables us to work with a well-defined power series instead of the path integral (actually this is the path integral in dimension 0):

\begin{dfn}[\cite{michiq}]
For a general formal action $\mathcal{S}{(x)}=-\frac{x^2}{2a}+V(x)\in x^2\mathbb{R}[[x]]$ we define the corresponding perturbative partition function to be the power series in $\hbar$ given by

\[\mathcal{F}[\mathcal{S}(x)](\hbar)=\sqrt{a}\sum_{n=0}^\infty (a\hbar)^n(2n-1)!![x^{2n}] e^{\frac{1}{\hbar}V(x)}.\]
\end{dfn}

This is a well-defined power series in $\hbar$ since the coefficient $[x^{2n}] e^{\frac{1}{\hbar}V(x)}$ is a polynomial in $\hbar^{-1}$ of degree less than $n$ because $V\in x^3\mathbb{R}[[x]]$. Just as the path integral, this map also has a diagrammatic meaning in terms of Feynman diagrams \cite{cvi}:

\begin{prop}
If $\mathcal{S}(x)=-\frac{x^2}{2a}+\sum_{k=3}^\infty \frac{\lambda_k}{k!}x^k$ with $a>0$, then 
\[\mathcal{F}[\mathcal{S}(x)](\hbar)=\sqrt{a} \sum_\Gamma \hbar^{|E(\Gamma)|-|V(\Gamma)|}\;\; \displaystyle\frac{a^{|E(\Gamma)|}\Pi_{v\in V(\Gamma)} \lambda_{n_v}}{|\mathrm{Aut}\;\Gamma|},\]

where the sum runs over all multigraphs $\Gamma$ in which the valency $n_v$ of every vertex $v$ is at least $3$, and where $|E(\Gamma)|$, $|V(\Gamma)|$, and $|\mathrm{Aut}\;\Gamma|$ are the sizes of the edge set, the vertex set, and the automorphism group of $\Gamma$ respectively. 
\end{prop}

So, in terms of Feynman diagrams, to compute the $n^\text{th}$ coefficient of $\mathcal{F}[\mathcal{S}(x)]$ we do the following 

\begin{enumerate}
    \item Draw all multigraphs $\Gamma$ with $|E(\Gamma)|-|V(\Gamma)|=n$. Note that this is one less than the loop number, it is the number of independent cycles in the graph (remember that independent cycles can be obtained by starting with a spanning tree and adding one edge at a time). The loop number is also known as the \textit{Betti number} of the graph. The number $|E(\Gamma)|-|V(\Gamma)|$ will be referred to as the \textit{excess} of $\Gamma$.
    
    \item Each vertex contributes with a factor that corresponds to its valency, that's how we get $\Pi_{v\in V(\Gamma)} \lambda_{n_v}$. Then we multiply with the factor $a^{|E(\Gamma)|}$. This process simply corresponds to the Feynman rules. The map that applies the Feynman rules will be denoted by $\phi_\mathcal{S}$.
    
    \item Divide by the size of the automorphism group of the graph. 
    \item Finally sum up all the contributions and multiply by $\sqrt{a}$.
\end{enumerate}

\begin{exm}[\cite{michiq}]
As an example, the action for $\varphi^3$-theory takes the form $\mathcal{S}(x)=-\frac{x^2}{2}+\frac{x^3}{3!}$. In that case 

\[Z^{\varphi^3}(\hbar)=\sum_{n=0}^\infty \hbar^n(2n-1)!! [x^{2n}]e^{\frac{x^3}{3!\hbar}}=\sum_{n=0}^\infty\displaystyle\frac{(6n-1)!!}{(3!)^{2n}(2n)!} \hbar^n.\]

In terms of Feynman diagrams only $3$-regular graphs will show up in $\varphi^3$-theory, hence we have 
\begin{align*}
Z^{\varphi^3}(\hbar)=\phi_{\mathcal{S}}&\Big(1+
\displaystyle\frac{1}{8}\;\raisebox{-0.2cm}{\includegraphics[scale=0.3]{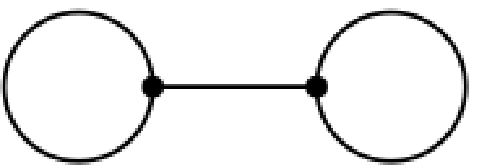}}+
\displaystyle\frac{1}{12}\;\raisebox{-0.2cm}{\includegraphics[scale=0.3]{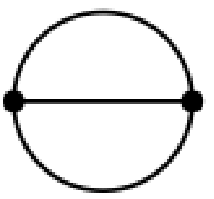}}+
\displaystyle\frac{1}{128}\;\raisebox{-0.2cm}{\includegraphics[scale=0.3]{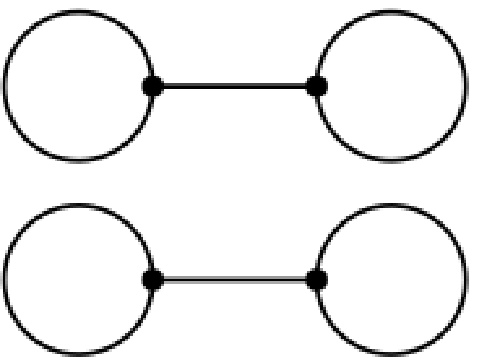}}+
\displaystyle\frac{1}{288}\;\raisebox{-0.2cm}{\includegraphics[scale=0.3]{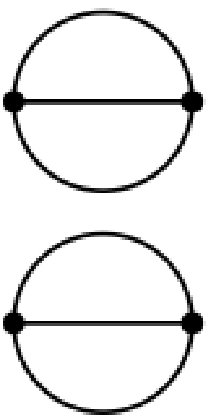}}+
\displaystyle\frac{1}{96}\;\raisebox{-0.2cm}{\includegraphics[scale=0.3]{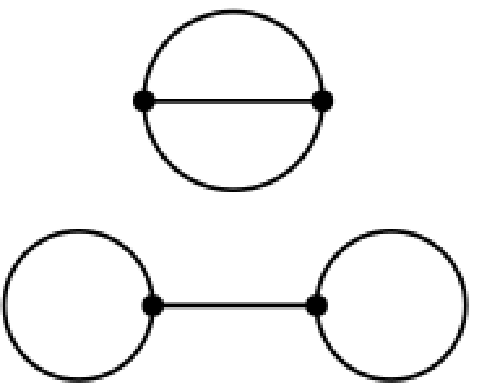}}+\\
&+\displaystyle\frac{1}{48}\;\raisebox{-0.2cm}{\includegraphics[scale=0.32]{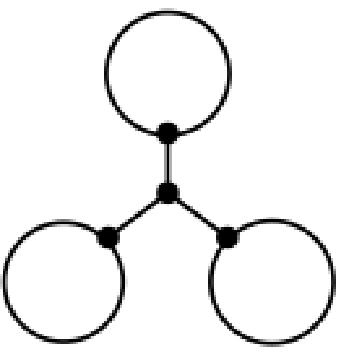}}+
\displaystyle\frac{1}{16}\;\raisebox{-0.2cm}{\includegraphics[scale=0.3]{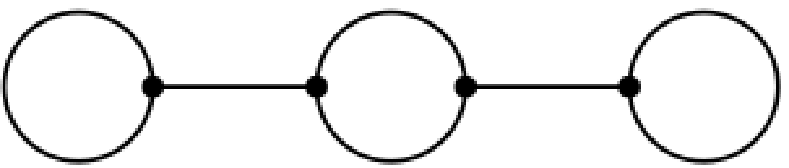}}+
\displaystyle\frac{1}{16}\;\raisebox{-0.2cm}{\includegraphics[scale=0.32]{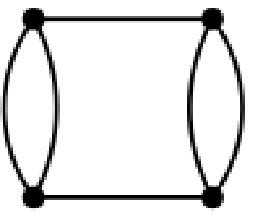}}+
\displaystyle\frac{1}{8}\;\raisebox{-0.2cm}{\includegraphics[scale=0.32]{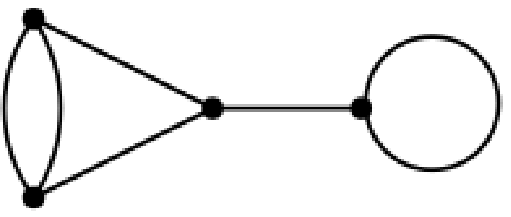}}+\\
&\\
&\displaystyle\frac{1}{24}\;\raisebox{-0.2cm}{\includegraphics[scale=0.32]{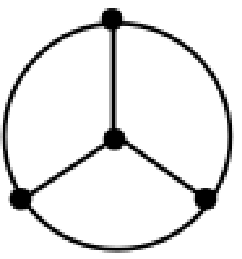}}+
\cdots\Big). \end{align*}

Then applying the Feynman rules  $\phi_\mathcal{S}$ does not change the coefficients in the sum since the contribution of any vertex is $1$ according to the described potential. Adding up the terms with the same loop number then gives 

\[Z^{\varphi^3}(\hbar)=1+\frac{5}{24}\hbar+\frac{385}{1152}\hbar^2+\cdots, \] which agrees with the first algebraic calculation.
\end{exm}

\subsection{Zero-Dimensional Scalar Theories with Interaction}

We will study such expansions that arise in QED theories, namely we shall consider quenched QED and Yukawa theory. These are examples of theories with interaction. It is impossible to completely cover the underlying physics, nevertheless we should be able to understand as much as needed for our purposes by anticipating the interrelations between the different entities defined. 

In the presence of interaction in the theory, the partition function takes the form 

\[Z(\hbar,j):=\int_\mathbb{R} \displaystyle\frac{1}{\sqrt{2\pi\hbar}} e^{\frac{1}{\hbar}\left(-\frac{x^2}{2}+V(x)+xj\right)}dx,\] 
where an additional term is added to the potential, namely $xj$, $j$ is called the \textit{source}.

With this extra term we can not directly expand the integral as we did before, but we can still achieve the same essence after a change of variables. Shift $x$ to $x+x_o$ where $x_o(j)$ is the unique power series solution to $x_o(j)=V'(x_o(j))+j$.  Then we get 

\begin{align*}
    Z(\hbar,j)&=\int_\mathbb{R} \displaystyle\frac{1}{\sqrt{2\pi\hbar}} e^{\frac{1}{\hbar}\left(-\frac{(x+x_o)^2}{2}+V(x+x_o)+(x+x_o)j\right)}dx\\
              &\\
              &=e^{\frac{1}{\hbar}\left(-\frac{x_o^2}{2}+V(x_o)+x_oj\right)}
              \int_\mathbb{R} \displaystyle\frac{1}{\sqrt{2\pi\hbar}} e^{\frac{1}{\hbar}\left(-\frac{x^2}{2}+V(x+x_o)-V(x_o)-xV'(x_o)\right)}dx\\
              &\\
              &=e^{\frac{1}{\hbar}\left(-\frac{x_o^2}{2}+V(x_o)+x_oj\right)}
              \mathcal{F}\left[-\displaystyle\frac{x^2}{2}+V(x+x_o)-V(x_o)-xV'(x_o)\right](\hbar).
\end{align*}

The exponential factor enumerates forests (collections of trees) with the corresponding conditions on vertices, these diagrams are referred to as the \textit{tree-level} diagrams. Tree-level diagrams contribute with negative powers of $\hbar$, and therefore we are going to isolate them so that the treatment for the main expansion remains clear. Remember that  Feynman diagrams are labeled, and so in order to restrict ourselves to connected diagrams we have to take the logarithm of the partition function:

\begin{align*}
    W(\hbar,j):&=\hbar \log 
    Z(\hbar,j)\\
    &=-\frac{x_o^2}{2}+V(x_o)+x_oj+\hbar \log\mathcal{F}\left[-\displaystyle\frac{x^2}{2}+
    V(x+x_o)-V(x_o)-xV'(x_o)\right](\hbar)
\end{align*}

generates all connected diagrams and is called the \textit{free energy}. Note that the extra $\hbar$ factor causes the powers to express the number of loops instead of the excess. Again we are using the notation in \cite{michiq} since we are eventually going to compare to parts of the work.

As customary in QFT, to move to the \textit{quantum effective action} $G$, which generates \textit{1PI} diagrams, one takes the Legendre transform of $W$:

\begin{align} \label{propergreenfnmichi}
    G(\hbar,\varphi_c):=W-j \varphi_c, 
\end{align}
where $\varphi_c:=\partial_jW$. The coefficients $[\varphi_c^n]G$ are called the \textit{(proper) Green functions} of the theory. 
Recall that from a graph theoretic point-of-view, being \textit{1PI} (1-particle irreducible)  is merely another way of saying $2$-connected. Thus, combinatorially, the Legendre transform, as in \cite{kjm2}, is seen to be the transportation from connected diagrams to $2$-connected or \textit{1PI} diagrams. In that sense, the order of the derivative  $\partial^n_{\varphi_c}G|_{\varphi_c=0}=[\varphi_c^n]G$ determines the number of external legs.

In the next part of the discussion we shall need the following physical jargon and terminology:

\begin{enumerate}
    \item The Green function $[\varphi_c^1]G=\partial_{\varphi_c}G|_{\varphi_c=0}$ generates all \textit{1PI} diagrams with exactly one external leg, which are called the \textit{tadpoles} of the theory (Figure \ref{tp}).
    
    \begin{figure}[h]
    \centering
    \includegraphics[scale=0.4]{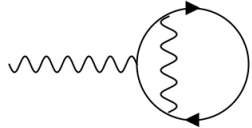}
    \caption{A tadpole diagram in QED}
    \label{tp}\end{figure}
    
    \item The Green function $[\varphi^2_c]G=\partial^2_{\varphi_c}G|_{\varphi_c=0}$ generates all \textit{1PI} diagrams with two external legs. Such a diagram is called a \textit{1PI propagator} (can replace an edge in the theory). 

    \begin{figure}[h]
    \centering
    \includegraphics[scale=0.4]{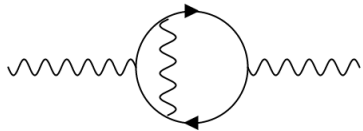}
    \caption{A propagator diagram}
    \label{propagator}
    \end{figure}
 
    \item For $n>2$, $\partial^n_{\varphi_c}G|_{\varphi_c=0}=[\varphi_c^n]G$ is called the \textit{$n$-point function}.
\end{enumerate}

In quenched QED, some of the quantities that we are going to compare their expansions with the generating series of $2$-connected chord diagrams are the \textit{renormalized} Green functions with respect to a chosen \textit{residue}. We shall therefore recall from Section \ref{hopfalg1PIsection} the basics of the Hopf-algebraic treatment of renormalization in the next section before proceeding into the real calculations. For more about this topic the reader can consult \cite{manchonhopf}, or the original paper by D. Kreimer and A. Connes \cite{conneskreimer}.


\subsection{Hopf-algebraic Renormalization}

 By the work in \cite{AliThesis,aless, renorm, kreimerr, kreimerrr, conneskreimer} we know that, for a given QFT theory, the superficially divergent 1PI Feynman graphs form a Hopf algebra $\mathcal{H}$. The product is defined to be the disjoint union, and the coproduct of a connected Feynman graph $\Gamma$ was defined according to Definition \ref{coprod}:
 
\[\Delta(\Gamma)=\underset{1PI \;\text{subgraphs}}{\underset{\gamma\;\text{product of divergent }}{\underset{\gamma\subseteq\Gamma}{\sum}}}\gamma\otimes\Gamma/\gamma,\]

and extended as an algebra morphism. The unit, counit, and antipode were denoted by $\mathbb{I}$, $\hat{\mathbb{I}}$, and $S$.

In \cite{karenbook} (and Section 3.1 in \cite{AliThesis})  we can see how the Dyson-Schwinger equations can be written in terms of the elements $X^r$, where $X^r$ was defined as 

\[X^r=1\pm \underset{\text{with residue}\;r }{\underset{\text{1PI graphs}\;\Gamma}{\sum}}\displaystyle\frac{1}{\text{Sym}\; \Gamma}\;\Gamma,\]

where the negative sign is assumed only when $r$ is edge-type. Also from \cite{karendiffeo} it is shown that if we use insertions in case of a theory with a single vertex type we get the equation:
\[ X^r=1\pm \sum_k B_+^{\gamma_{r,k}}(X^rQ^k),\]

where the sum is over all primitive 1PI diagrams with loop number $k$ and residue $r$, and where $Q$ is the invariant charge as defined in Section \ref{invariantch}.

The identity 

\begin{equation}\label{most}
    \Delta X^r=\sum_{L=0}^\infty\left. X^r Q^L\otimes X^r\right|_L,
\end{equation}

is of most importance in the context of renormalization \cite{kreimerB}. The
 $|_L$, as used in \cite{michiq}, is the restriction of the sum to graphs with loop number $L$.

We have seen in Section \ref{hopfalg1PIsection} that the Feynman rules are simply characters from $\mathcal{H}$ to a commutative algebra $A$. For zero-dimensional field theories the Feynman rules will be $\phi:\mathcal{H}\longrightarrow\mathbb{R}[[\hbar]]$:
\begin{equation}\label{FFFF}
    \phi\{\Gamma\}(\hbar)=\hbar^{\ell(\Gamma)},
\end{equation}
where we follow the notation in \cite{michiq} for putting the arguments from $\mathcal{H}$ in curly brackets. 

In that case, the Green functions, or the generating function of $1PI$ Feynman graphs with residue $r$ are defined as 

\begin{equation}\label{g^r}
    g^r(\hbar):=\phi\{X^r\}(\hbar)=1\pm \underset{\text{with residue}\;r }{\underset{\text{1PI graphs}\;\Gamma}{\sum}}\displaystyle\frac{\hbar^{\ell(\Gamma)}}{\text{Sym}\; \Gamma}, 
\end{equation}
where $\ell(\Gamma)$ is the loop number as before. If residue $r$ is the $k$ external legs residue, then $g^r=\partial^k_{\varphi_c}G|_{\varphi_c=0}$ , the $k$th derivative of the quantum effective action. 

In our case of zero-dimensional QFT, the fact that the target algebra for the Feynman rules is $\mathbb{R}[[\hbar]]$  limits the choice for a Rota-Baxter operator\footnote{Remember that a Rota-Baxter operator $R$  is used in the renormalization scheme to extract (in terms of an induced Birkhoff decomposition) the divergent part of the integral. See Section \ref{hopfalg1PIsection}.}
$R:\mathbb{R}[[\hbar]]\longrightarrow\mathbb{R}[[\hbar]]$ that respects the grading of $\mathcal{H}$. The only choice for a meaningful renormalization scheme in this case is $R=\text{id}$ (see \cite{michiq}).

Thus, by our definitions in Section \ref{hopfalg1PIsection} (equation \ref{S_R}), the counterterm map for the renormalization scheme $R=\text{id}$ is given by

\[S^\phi=R\circ\phi\circ S=\phi\circ S.\] Then the renormalized Feynman rules is 

\begin{align}\label{tobeusednow}
\phi_{\text{ren}}:=S_R^\phi\ast\phi=S^\phi\ast\phi=(\phi\circ S)\ast\phi,   
\end{align} 
where $\ast$ is the convolution product (Definition \ref{convo}).
However, the action of the last expression on an arbitrary element of $\mathcal{H}$ is:

\begin{align*}\label{tobeusednow}
((\phi\circ S)\ast\phi)(\Gamma)
&=m\circ((\phi\circ S)\otimes\phi)\circ\Delta(\Gamma)&\\
&=\underset{1PI \;\text{subgraphs}}{\underset{\gamma\;\text{product of divergent }}{\underset{\gamma\subseteq\Gamma}{\sum}}} \phi(S(\gamma))\phi(\Gamma/\gamma)&\\
&&\\
&=\phi\Big(\sum S(\gamma)(\Gamma/\gamma)\Big)&
 \text{(since $\phi$ is a character)}\\
&=\phi(\mathbb{I}(\hat{\mathbb{I}}(\Gamma))& \text{(by definition of the antipode $S$),}
\end{align*}
which is zero for all nonempty elements in $\mathcal{H}$ since $\mathbb{I}\circ\hat{\mathbb{I}}$ maps all elements in $\mathcal{H}$ to zero except for the empty graph, which is mapped to itself.

Thus, equation (\ref{tobeusednow}) becomes
\begin{equation}\label{renphi}
    \phi_{\text{ren}}=\mathbb{I}\circ\hat{\mathbb{I}},
\end{equation}
and whence the renormalized Green functions are
\begin{equation}\label{reng}
    g^r_{\text{ren}}=\phi_{\text{ren}}\{X^r\}(\hbar)=1\pm0=1.
\end{equation}

Note that in \cite{michiq} the signs are different since, as mentioned earlier, $X^r$ in their convention is $-1$ times ours.
 Finally, what we will care for the most are the counterterms:
\begin{equation}\label{countertermsz_r}
    z_r:=S^\phi\{X^r\}.
\end{equation}
Notice that since $\mathcal{H}$ is commutative, the definition of the convolution product together with equation (\ref{most}) now yield
\begin{align}
    1
    &=\phi_{\text{ren}}\{X^r\}(\hbar)\\\nonumber
    &=(m\circ(S^\phi\otimes\phi)\circ \Delta X^r)(\hbar)\\\nonumber
    &=(m\circ(\phi\otimes S^\phi)\circ \Delta X^r)(\hbar)\\\nonumber
    &=(m\circ(\phi\otimes S^\phi)\circ (\sum_{L=0}^\infty\left. X^r Q^L\otimes X^r\right|_L))(\hbar)\\\nonumber
    &=\sum_{L=0}^\infty \phi\{X^r\}(\hbar)\;\;(\phi\{Q\}(\hbar))^L\;\;[\hbar^L] S^\phi\{X^r\}(\hbar)\\\nonumber
    &= \phi\{X^r\}(\hbar)\;\;\;S^\phi\{X^r\}(\hbar \phi\{Q\}(\hbar))\\\nonumber
    &=g^r(\hbar) \;z_r(\hbar \alpha(\hbar)).
\end{align}

Let $\hbar(y)$ be the unique power series solution of $y=\hbar(y)\alpha(\hbar(y))$. In \cite{michiq}, $y$ is called the \textit{renormalized expansion parameter} and is denoted by $\hbar_{\text{ren}}$. Then, substituting in equation  we get

\begin{equation}
    z_r(\hbar_{\text{ren}})=\displaystyle\frac{1}{g^r(\hbar(\hbar_{\text{ren}}))}.\end{equation}

The following result was proven by M. Borinsky in \cite{michilattice}, and we shall depend on it in the combinatorial treatment in the next section.

\begin{thm}[\cite{michilattice}]\label{theorem important zr}
In a theory with a cubic vertex-type, the numeric coefficients in $z_r(\hbar_{\text{ren}})$ count the number of primitive diagrams if $r$ is vertex-type.
\end{thm}


\subsection{QED Theories, Quenched QED, and Yukawa Theory}

The two theories that we are concerned with here are quenched QED and Yukawa theory, which are examples of QED-type theories. In these theories we have two particles: fermion and boson (wiggly and dashed edges) particles, and we have only three-valent vertices of the type fermion-fermion-boson. We will compute the asymptotics of $\;z_{\phi_c|\psi_c|^2}(\hbar_{\text{ren}})\;$ in quenched QED, as well as the asymptotics of the green functions $\left.\partial^i_{\phi_c}(\partial_{\psi_c}\partial_{\bar{\psi}_c})^j\;G^{\text{Yuk}}\right|_{\phi_c=\psi_c=0}$. Our approach is completely combinatorial and depends on establishing bijections between the diagrams in the combinatorial interpretation of the considered series and different classes of chord diagrams. Unlike the approach applied in \cite{michiq}, we do not need to refer to singularity analysis nor the representation of $\mathcal{S}(x)$ by affine hyperelliptic curves.

\subsection{The Partition Function}\label{partitionfunctionsection}
The partition function takes the form 
\begin{equation}
 Z(\hbar,j,\eta)=\int_\mathbb{R} \displaystyle\frac{1}{\sqrt{2\pi\hbar}}\; e^{\frac{1}{\hbar}\left(-\frac{x^2}{2}+jx+\frac{|\eta|^2}{1-x}+\hbar \log \frac{1}{1-x}\right)}dx .\label{Yu}\end{equation}

We are not going to discuss the physical reasoning behind the above expression, the reader may refer to QFT books or surveys for more details, e.g. see \cite{michi}. We only hint that, combinatorially, $\hbar \log \frac{1}{1-x}$ generates fermion loops, while $\frac{|\eta|^2}{1-x}$ generates a fermion propagator. The special examples of Yukawa theory and quenched QED will be as follows:

\begin{enumerate}
    \item Quenched QED is an approximation of QED where fermion loops are not present. So that the term $\hbar \log \frac{1}{1-x}$ does not appear in the partition function. Thus, the partition function for quenched QED is given by 

\[Z^{QQED}(\hbar,j,\eta)=\int_\mathbb{R} \displaystyle\frac{1}{\sqrt{2\pi\hbar}}\; e^{\frac{1}{\hbar}\left(-\frac{x^2}{2}+jx+\frac{|\eta|^2}{1-x}\right)}dx.\]
    
    \item For zero-dimensional Yukawa theory the partition function is just the integral in equation (\ref{Yu}). That is, the partition function for zero-dimensional Yukawa theory is given by 

\[Z^{Yuk}(\hbar,j,\eta)=\int_\mathbb{R} \displaystyle\frac{1}{\sqrt{2\pi\hbar}}\; e^{\frac{1}{\hbar}\left(-\frac{x^2}{2}+jx+\frac{|\eta|^2}{1-x}+\hbar \log\frac{1}{1-x}\right)}dx.\]
\end{enumerate}

\section{Quenched QED}

For this theory we are interested in the asymptotics of the counterterm $\;z_{\phi_c|\psi_c|^2}(\hbar_{\text{ren}})\;$ obtained in \cite{michiq} (page 38). By Theorem \ref{theorem important zr}, since QQED has only one type of vertices which is three-valent, this series enumerates the number of primitive quenched QED diagrams with vertex-type residue (see sequence \href{https://oeis.org/A049464}{A049464} of the OEIS for the first entries).

Thus, this is the same as counting the number of all diagrams $\gamma$ with the following specifications:
\begin{enumerate}
    \item two types of edges, fermion and boson (photon) edges, represented as \raisebox{-0cm}{\includegraphics[scale=0.68]{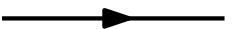}} and \raisebox{-0.23cm}{\includegraphics[scale=0.3]{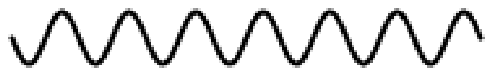}},  respectively;
    \item only three-valent vertices with the structure \raisebox{-0.64cm}{\includegraphics[scale=0.35]{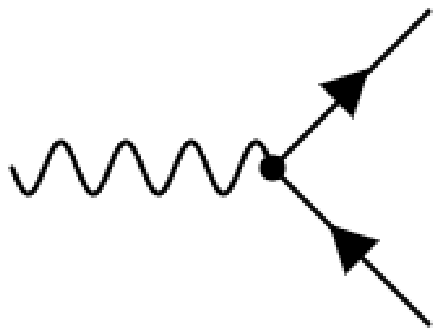}}, with one fermion in, one fermion out, and one photon;
    
    \item no fermion loops;
    \item the residue $\text{res} (\gamma)$ is vertex-type; and
    \item $\gamma$ is $1PI$ primitive, in other words it is edge-connected and contains no subdivergences (Definition \ref{primitivediags} and Definition \ref{subdiv}).
\end{enumerate}
We let $\mathcal{Q}$ be the class of all such diagrams.

\begin{exm}\label{exampleQQED diag}
The following diagram in Figure \ref{fig:my_label} is in $\mathcal{Q}$, whereas the diagrams in Figure \ref{notqqeddiags} are not in $\mathcal{Q}$.
\begin{figure}[h]
    \centering
    \includegraphics[scale=0.4]{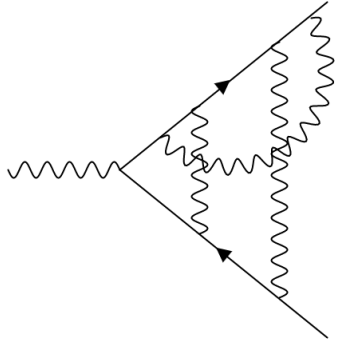}
    \caption{An example of a quenched QED diagram in $\mathcal{Q}$}
    \label{fig:my_label}
\end{figure}

\begin{figure}[h]
    \centering
    \includegraphics[scale=0.4]{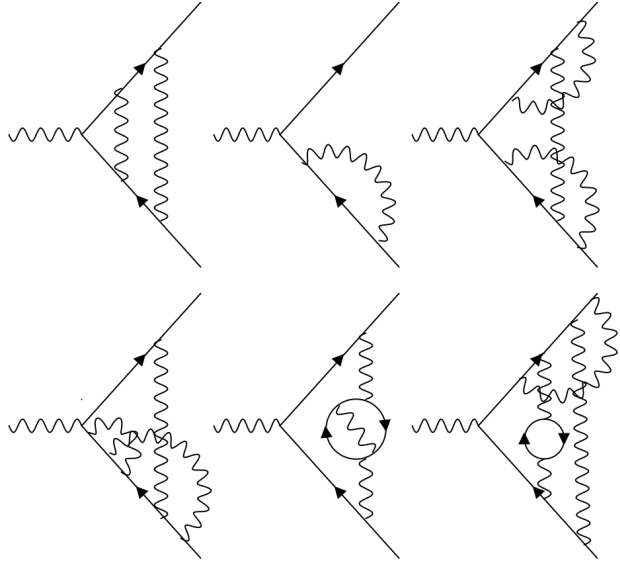}
    \caption{Diagrams not in $\mathcal{Q}$.}
    \label{notqqeddiags}
\end{figure}

In Figure \ref{notqqeddiags}, the second diagram is not 1PI, the rest of the first four diagrams are all 1PI, but they are not  primitive. The last two diagrams are not in $\mathcal{Q}$ and are not even quenched as they contain fermion loops. 

\end{exm}

\begin{thm}\label{myresultinquenched}
The generating series $\;z_{\phi_c|\psi_c|^2}(\hbar_{\text{ren}})\;$ and $\;z_{|\psi_c|^2}(\hbar_{\text{ren}})\;$ count $2$-connected chord diagrams. More precisely, $$[\hbar_{\text{ren}}^{n-1}]\;z_{\phi_c|\psi_c|^2}(\hbar_{\text{ren}})=
[\hbar_{\text{ren}}^{n}]\;z_{|\psi_c|^2}(\hbar_{\text{ren}})=[x^n]\;C_{\geq2}(x).$$
\end{thm}

\begin{proof}
As we mentioned above, the class of diagrams counted by $\;z_{\phi_c|\psi_c|^2}(\hbar_{\text{ren}})\;$ is to be denoted by $\mathcal{Q}$. Thus, $\mathcal{Q}$ consists of 1PI primitive quenched QED diagrams with two external fermion legs and one photon leg. By definition, every graph in $\mathcal{Q}$ has a (wiggly) photon external leg $r$, and two directed fermion external legs $f_1$ and $f_2$. 
We now start by proving the following claim:

\underline{\textbf{Claim 1:}} If $\Gamma$ is a graph in $\mathcal{Q}$, then there exists a unique fermion-only path $P$ from $f_1$ to $f_2$. Moreover, $P$ passes through the vertex at $r$ and every vertex in the graph is on $P$. In addition, the loop number in $\Gamma$ is equal to the number of internal photon edges, and so either is counted by the power of $\hbar_{\text{ren}}$ in the series.

\textbf{Proof:} Generally, if we remove all photon edges from a  graph with the 3-valent vertex residue we should  get a single directed path of fermion edges (because otherwise we will have more than 2 external fermion legs if the photon edges are restored) and a set fermion loops. Now, in our case, we can only get the path, which we denote by $P$ and which should then carry all the vertices in the original graph. In particular, the number of vertices in a graph $\Gamma\in\mathcal{Q}$ will be 1+the number of internal fermion edges. 

To see that the rest of the claim is indeed true first recall Euler's formula \[|V(\Gamma)|-|E(\Gamma)|+\ell(\Gamma)=1,\]
where as usual $|V(\Gamma)|$ is the number of vertices, $|E(\Gamma)|$ is the number of internal edges, and $\ell(\Gamma)$ is the number of loops or independent cycles in $\Gamma$.

Note that the external legs do not alter this relation. We have two types of edges, photons and fermions, so let us assume that $p$ is the number of internal photon edges and $f$ is the number of internal fermion edges, thus $|E(\Gamma)|=p+f$. Now, the most useful observation is that, in our case, we  should have $p=\ell(\Gamma)$. Indeed, we have seen that 
\[f=|V(\Gamma)|-1,\]
from which it follows that $p=\ell(\Gamma)$.
\textbf{This proves Claim 1.}

So, we can generally think of graphs in $\mathcal{Q}$ as in the figure below,
where $P$ is the unique path formed by all  directed fermion edges. $P$ goes from $f_1$ to $f_2$ and passes through the vertex at $r$. All vertices lie on $P$.
\begin{center}

\raisebox{0cm}{\includegraphics[scale=0.4]{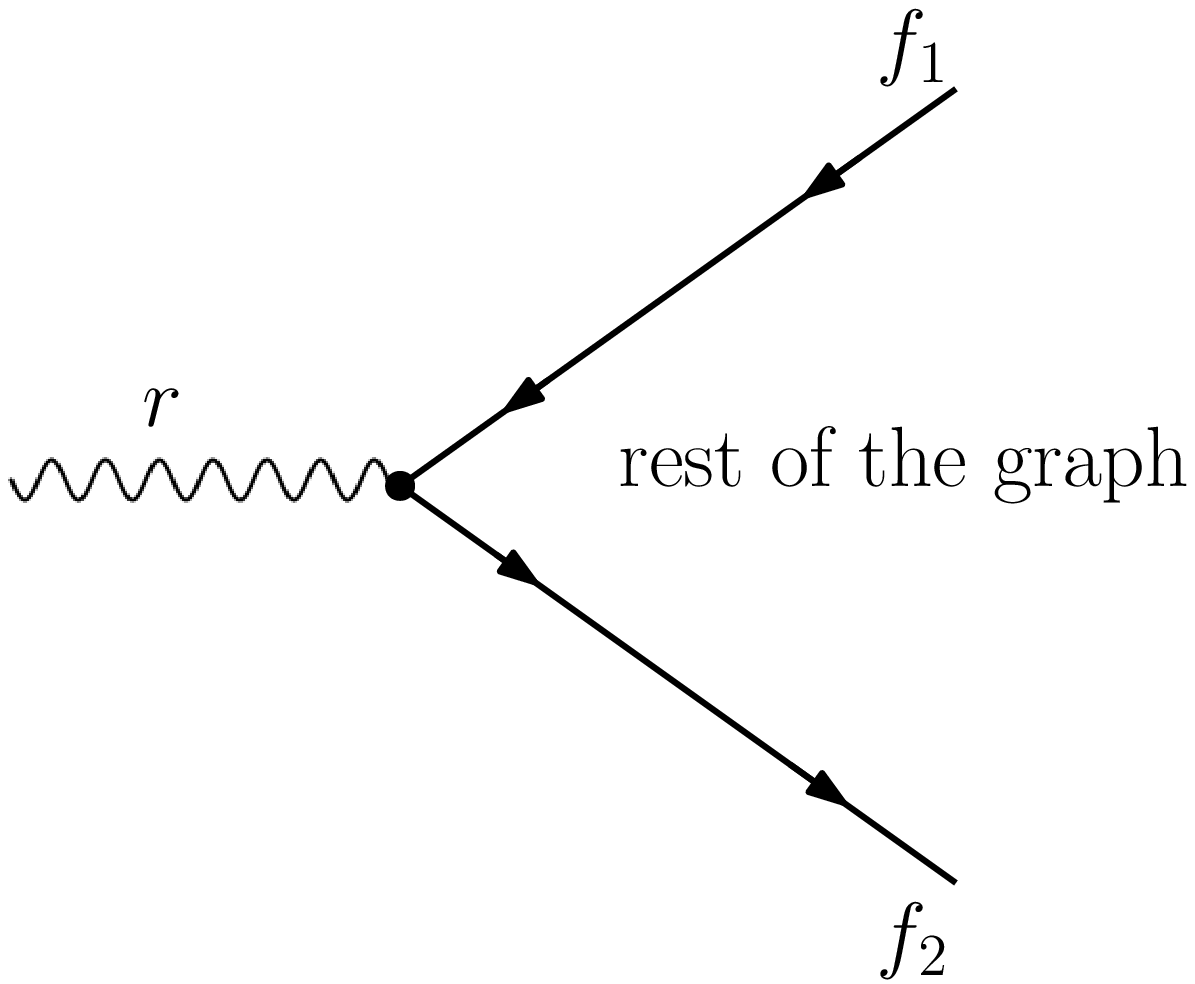}}
    
\end{center}

This means we can uniquely put any graph $\Gamma\in\mathcal{Q}$ in the form of a rooted chord diagram, namely by straightening $P$. See Figure \ref{intochords1} for an example. 
\begin{figure}[h]
    \centering
    \includegraphics[scale=0.42]{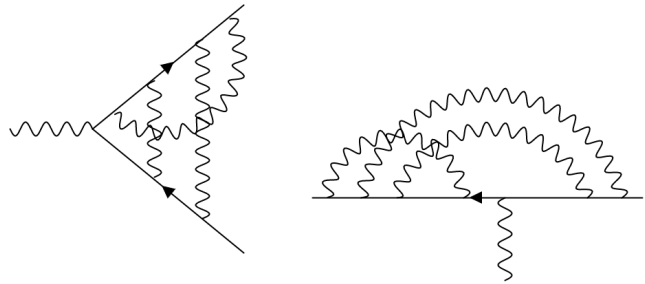}
    \caption{A primitive quenched QED graph and its representation as a chord diagram}
    \label{intochords1}
\end{figure}

For simplicity of drawing we shall now and forth in the proof use dashed or light lines for photons and drop the direction on the fermion edges on $P$. Also, let us agree that, in the chord diagram representation, we will bring $r$ to the front to play the role of a root, and still carry the information for the external leg position at its other end. Thus, for example, the graph in Figure \ref{intochords1} is now represented as follows:
\begin{center}

\raisebox{0cm}{\includegraphics[scale=0.42]{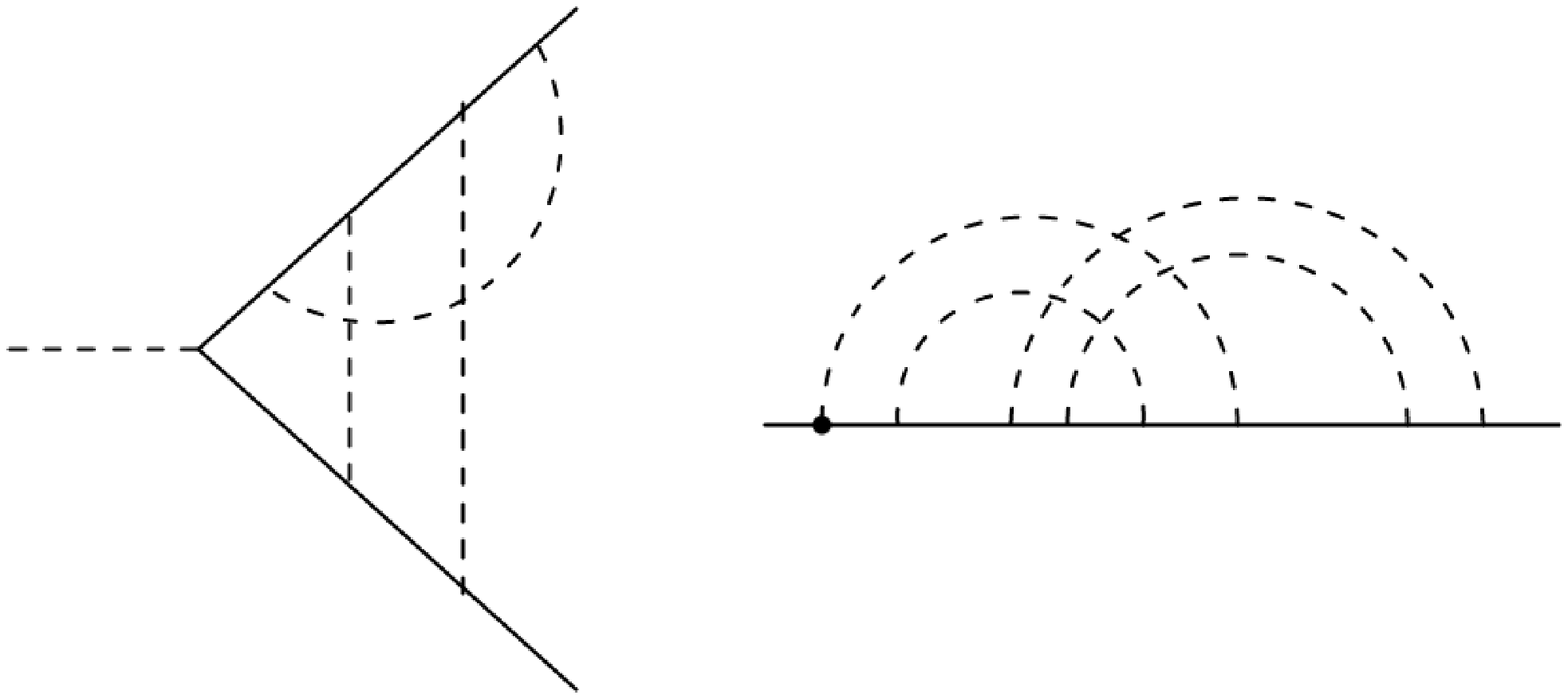}}
    
\end{center}
The chord diagram representation of $\Gamma\in\mathcal{Q}$ will be denoted by $C(\Gamma)$. Let us also denote the right end of the root $r$ by $v$. The only property of $\mathcal{Q}$ that we still haven't used is that a graph in $\mathcal{Q}$ is primitive.

\underline{\textbf{Claim 2:}} A 1PI quenched QED graph $\Gamma$ is primitive if and only if it is $2$-connected in the chord diagram representation. Subdivergences are translated into either a disconnection or a bridge.

\textbf{Case 1:} Assume that $C(\Gamma)$ is disconnected. This means that there exists an isolated component of chords to the right or left of $v$. On the original graph this is simply a propagator-type subdivergence inserted on one of the fermion edges. For an example see Figure \ref{disconnectandsubdiv} below.
\begin{figure}[h!]
    \centering
\raisebox{0cm}{\includegraphics[scale=0.55]{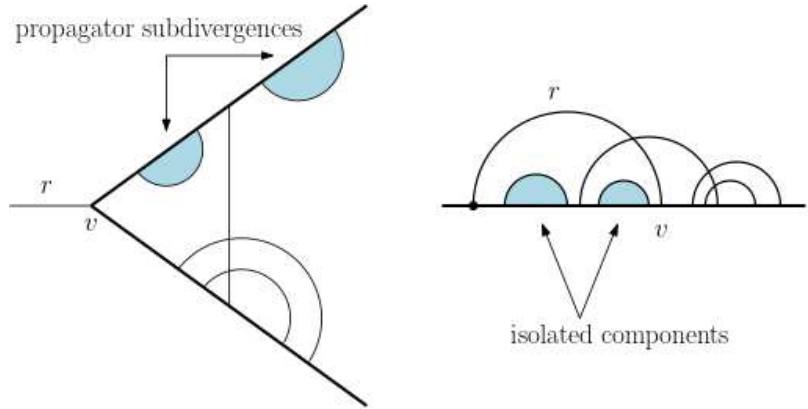}}
    \caption{Disconnections and propagator subdivergences.}\label{disconnectandsubdiv}
\end{figure}
The converse is also clearly true, a propagator subdivergence is translated into an isolated component in $C(\Gamma)$.

\textbf{Case 2: } Assume that $C(\Gamma)$ has a reason $S$ for connectivity-1, in the sense of Definition \ref{connectv1dfn}. Then the cut $c$ for $S$ is either the root chord $r$ or not.

(A) If $c=r$, then in $\Gamma$, $S$ together with $r$ correspond to a vertex-type subdivergence inserted at the vertex of the photon edge $r$.

(B) If $c\neq r$, then $S$ lies to the right or left of $v$ in $C(\Gamma)$. On $\Gamma$, this is a vertex-type subdivergence inserted an end of the photon edge $c$. 

Conversely, by the same means, every vertex-type subdivergence in $\Gamma$ gives a reason for connectivity-1 in $C(\Gamma)$. \textbf{This proves Claim 2}. Figure \ref{connect1andsubdiv} illustrates  situations  (A) and (B) on one and the same graph.
\begin{figure}[h!]
    \centering
\raisebox{0cm}{\includegraphics[scale=0.6]{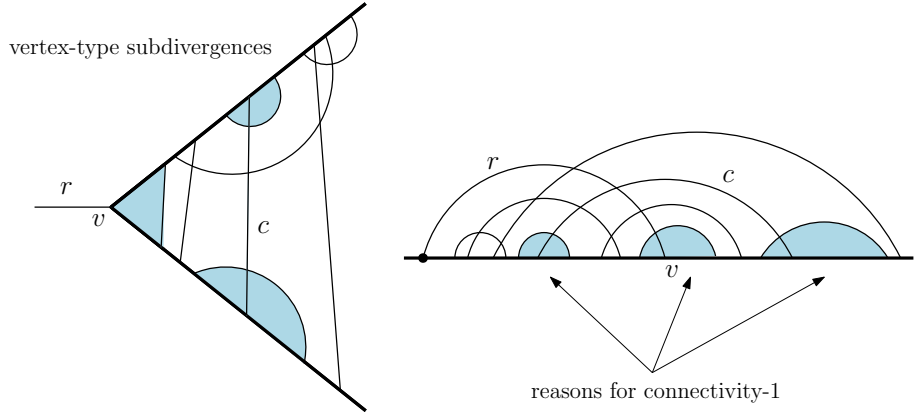}}
    \caption{Connectivity-1 and vertex-type subdivergences.}\label{connect1andsubdiv}
\end{figure}

Thus, every graph in $\mathcal{Q}$ is uniquely represented as a $2$-connected chord diagram with the number of chords equal to the number of internal photon edges and also equal to the loop number of the graph. The generating series $\;z_{|\psi_c|^2}(\hbar_{\text{ren}})\;$ counts the same diagrams, it only differs in not having the external photon leg $r$. The removal of the  external leg $r$ will not change the argument above: propagator-type subdivergences correspond to isolated components in the chord diagram representation and vertex-type subdivergences correspond to reasons for connectivity-1. This completes the proof.
\end{proof}
This gives 
$ \mathcal{A}^2_{\frac{1}{2}}z_{\phi_c|\psi_c|^2}(\hbar_{\text{ren}})=
\frac{1}{\hbar_{\text{ren}}}\mathcal{A}^2_{\frac{1}{2}}z_{|\psi_c|^2}(\hbar_{\text{ren}})=
\mathcal{A}^2_{\frac{1}{2}}C_{\geq2}(x),
$
and by equation (\ref{computC2asympt}):
\begin{align}
    [\hbar_{\text{ren}}^{n-1}]\;&z_{\phi_c|\psi_c|^2}(\hbar_{\text{ren}})=
[\hbar_{\text{ren}}^{n}]\;z_{|\psi_c|^2}(\hbar_{\text{ren}})=[x^n]\;C_{\geq2}(x)=\nonumber\\
    & = e^{-2}  \bigg((2n-1)!!-6(2n-3)!!-4(2n-5)!!-
    \displaystyle\frac{218}{3}(2n-7)!!- \nonumber\\
    &\; \qquad       -890(2n-9)!!-\displaystyle\frac{196838}{15}(2n-11)!!-\cdots\bigg)
    \nonumber\\
    & = e^{-2} (2n-1)!!\bigg(1-\displaystyle\frac{6}{2n-1}-\displaystyle\frac{4}{(2n-3)(2n-1)}-\displaystyle\frac{218}{3}\displaystyle\frac{1}{(2n-5)(2n-3)(2n-1)}- \nonumber\\
    &\;   \qquad     -\displaystyle\frac{890}{(2n-7)(2n-5)(2n-3)(2n-1)}-\displaystyle\frac{196838}{15}\displaystyle\frac{1}{(2n-9)\cdots(2n-1)}-\cdots\bigg),
\end{align}
which coincides with the result in \cite{michiq}.

\appendix
\section{Factorially Divergent Power Series}\label{factorially}

 This section aims to provide the necessary background for \textit{factorially divergent power series}, as introduced in Chapter 4 in \cite{michi}. In \cite{michi, michi1}, M. Borinsky  studied sequences $a_n$ whose asymptotic behaviour for large $n$ follows a relation like

\begin{equation}
    a_n=\alpha^{n+\beta} \Gamma(n+\beta)\bigg(c_0+\displaystyle\frac{c_1}{\alpha(n+\beta-1)}+\displaystyle\frac{c_2}{\alpha^2(n+\beta-1)(n+\beta-2)}+\cdots\bigg),
\end{equation}

where $\alpha\in\mathbb{R}_{>0}$, and $\beta, c_k \in \mathbb{R}$. We will need to use the usual big and small o-notation for asymptotic analysis: Given a sequence $a_n$, $\mathcal{O}(a_n)$ will denote the class of sequences $b_n$ satisfying $\limsup_{n\rightarrow\infty}|\frac{b_n}{a_n}|<\infty$; whereas $o(a_n)$ shall denote the sequences $b_n$ such that $\lim_{n\rightarrow\infty}\frac{b_n}{a_n}=0$. Moreover, $a_n=b_n+\mathcal{O}(c_n) $ should mean that $a_n-b_n\in\mathcal{O}(c_n)$. Following \cite{michi}, we adopt the notation $\Gamma_\beta^\alpha(n):=\alpha^{n+\beta}\Gamma(n+\beta)$, where $\Gamma(z)=\int_0^\infty x^{z-1}e^{-x}dx$ for $\text{Re}(z)>0$ is the gamma function.

\begin{dfn}[Factorially Divergent Power Series]\label{fdps} For real numbers $\alpha$ and $\beta$, with $\alpha>0$, the subset $\mathbb{R}[[x]]_\beta^\alpha$ of $\mathbb{R}[[x]]$ will denote the set of all formal power series $f$ for which there exists a sequence $(c_k^f)_{k\in \mathbb{N}}$ of real numbers such that

\begin{equation}
    f_n=\overset{R-1}{\underset{k=0}{\sum}}c_k^f \Gamma_\beta^\alpha(n-k)+\mathcal{O}(\Gamma_\beta^\alpha(n-R)),\;\;\text{for all} \;R\in\mathbb{N}_0 \;\;\; \text{(positive integers).}\label{asy}\end{equation}
 \end{dfn}

\begin{rem}
From the definition it follows that $\mathbb{R}[[x]]_\beta^\alpha$ is a linear subspace of $\mathbb{R}[[x]]$. 
\end{rem}

\begin{rem}\label{notinj}
Also, by the above definition all real power series with a non-vanishing radius of convergence belong to $\mathbb{R}[[x]]_\beta^\alpha$, with $c_k^f=0$ for all $k$ since in this case $f_n=o(\Gamma^\alpha_\beta(n-R))$ for all $R\in \mathbb{N}_0$.
\end{rem}

The following proposition also follows directly from the definition.

\begin{prop}[\cite{michi}, Ch.4]
 The sequence   $(c_k^f)_{k\in \mathbb{N}}$ is unique for every $f\in\mathbb{R}[[x]]_\beta^\alpha$; actually $c_N^f=\lim_{n\rightarrow\infty}\frac{f_n-\sum^{N-1}_{k=0}c_k^f \Gamma_\beta^\alpha(n-k)}{(\Gamma_\beta^\alpha(n-N))}$ for $N\in\mathbb{N}_0$.
\end{prop}

\begin{prop}[\cite{michi}, Prop 4.3.1]\label{subring}
Given $\alpha,\beta\in \mathbb{R}$, with $\alpha>0$, the set $\mathbb{R}[[x]]_\beta^\alpha$ is a subring of $\mathbb{R}[[x]]$.
\end{prop}

Note that  the identity in (\ref{asy}) stands for an asymptotic expansion with asymptotic scale $\alpha^{n+\beta}\Gamma(n+\beta)$ (refer to \cite{asymptotic} for a detailed literature on the topic). The ring $\mathbb{R}[[x]]_\beta^\alpha$ is referred to as \textit{a ring of factorially divergent power series}. Now, given $f\in\mathbb{R}[[x]]_\beta^\alpha$, we can associate the coefficients $(c_k^f)_{k\in \mathbb{N}}$ of the asymptotic expansion  with a new ordinary power series:

\begin{dfn}[\cite{michi1}]\label{map}
For $\alpha,\beta\in \mathbb{R}$, with $\alpha>0$, let $ \mathcal{A}_\beta^\alpha:\mathbb{R}[[x]]_\beta^\alpha\rightarrow\mathbb{R}[[x]]$ be the map that has the following action for every $f\in\mathbb{R}[[x]]_\beta^\alpha$
\[(\mathcal{A}_\beta^\alpha f)(x)=\overset{\infty}{\underset{k=0}{\sum}}c_k^fx^k.\]
\end{dfn}

This interpretation allows us to use the coefficients $c_k^f$ to get various results as we will see later. In \cite{michi} M. Borinsky provides an extensive analysis for the map $\mathcal{A}_\beta^\alpha $, we will include some of the properties without proof and will try to include proofs only as much as needed.

\begin{rem}

A map of this type is called an \textit{alien derivative (operator)}  in the context of resurgence theory \cite{resurgence}. We will use this terminology occasionally.
\end{rem}

\begin{rem}

Form the definition we see that $\mathcal{A}_\beta^\alpha$ is linear. $\mathcal{A}_\beta^\alpha$ is not injective since it vanishes for power series with nonzero radius of convergence as mentioned above in Remark \ref{notinj}.
\end{rem}

\begin{prop}[\cite{michi}, Prop 4.1.1]\label{plusm1} For $m\in\mathbb{N}_0$,
$f\in \mathbb{R}[[x]]_\beta^\alpha$ if and only if $f\in\mathbb{R}[[x]]_{\beta+m}^\alpha$ and $\mathcal{A}_{\beta+m}f\in x^m\mathbb{R}[[x]]$. In this case $x^m\big(\mathcal{A}_\beta^\alpha f\big)(x)=\big(\mathcal{A}_{\beta+m}^\alpha f\big)(x)$.
\end{prop}

\begin{proof}
First note that \[\Gamma^\alpha_\beta(n)=\alpha^{n-m+\beta+m}\Gamma(n-m+\beta+m)=\Gamma_{\beta+m}^\alpha(n-m).\]

Now, 
$f_n=\overset{R-1}{\underset{k=0}{\sum}}c_k^f \Gamma_{\beta}^\alpha(n-k)+\mathcal{O}(\Gamma_{\beta}^\alpha(n-R)),\;\;\text{for all} \;R\in\mathbb{N}_0 ,$ can be re-indexed as 
\[f_n=\overset{R'-1}{\underset{k=m}{\sum}}c_{k-m}^f \Gamma_{\beta}^\alpha(n-k+m)+\mathcal{O}(\Gamma_{\beta+m}^\alpha(n-R')),\;\;\text{for all} \;R'\geq m.\]

By the observation at the beginning the latter is equivalent to
\[f_n=\overset{R'-1}{\underset{k=m}{\sum}}c_{k-m}^f \Gamma_{\beta+m}^\alpha(n-k)+\mathcal{O}(\Gamma_{\beta+m}^\alpha(n-R')),\;\;\text{for all} \;R'\geq m.\]

Which proves the first part of the statement. Also, in that case

$\big(\mathcal{A}_{\beta+m}^\alpha f\big)(x)=\sum_{k=m}^\infty c_{k-m}^fx^k=x^m\big(\mathcal{A}_{\beta}^\alpha f\big)(x)\in x^m\mathbb{R}[[x]]$.
\end{proof}

The following corollary now follows.
\begin{cor}\label{corplusm1}
For all $m\in\mathbb{N}_0$,
$\mathbb{R}[[x]]_\beta^\alpha \subset \mathbb{R}[[x]]_{\beta+m}^\alpha$. 
\end{cor}

Thus we can always assume that $\beta>0$, which is very convenient in deriving many results for the ring $\mathbb{R}[[x]]_\beta^\alpha$ \cite{michi}.

\begin{prop}[\cite{michi}, Prop 4.1.2]\label{plusm2} For $m\in\mathbb{N}_0$,
$f\in \mathbb{R}[[x]]_\beta^\alpha \cap x^m\mathbb{R}[[x]]$ if and only if $\displaystyle\frac{f(x)}{x^m}\in\mathbb{R}[[x]]_{\beta+m}^\alpha$. In this case $\big(\mathcal{A}_\beta^\alpha f\big)(x)=\big(\mathcal{A}_{\beta+m}^\alpha \displaystyle\frac{f(x)}{x^m}\big)(x)$.
\end{prop}

\begin{proof}
Again, since \[\Gamma^\alpha_\beta(n+m)=\alpha^{n+\beta+m}\Gamma(n+\beta+m)=\Gamma_{\beta+m}^\alpha(n),\]
 we can argue as follows:
 
 The `only if part' follows by Proposition \ref{plusm1}. For the `if' part, assume that $g(x)=\displaystyle\frac{f(x)}{x^m}\in\mathbb{R}[[x]]_{\beta+m}^\alpha$. This gives that 

\[f_{n+m}=g_n=\overset{R-1}{\underset{k=0}{\sum}}c_k^g \Gamma_{\beta+m}^\alpha(n-k)+\mathcal{O}(\Gamma_{\beta+m}^\alpha(n-R)),\;\;\text{for all} \;R\in\mathbb{N}_0.\] 

By the observation above this is equivalent to
\[f_{n+m}=\overset{R-1}{\underset{k=0}{\sum}}c_k^g \Gamma_{\beta}^\alpha(n+m-k)+\mathcal{O}(\Gamma_{\beta}^\alpha(n+m-R)),\;\;\text{for all} \;R\in\mathbb{N}_0.\] 

Thus, for all $n\geq m$,
\[f_n=\overset{R-1}{\underset{k=0}{\sum}}c_k^g \Gamma_{\beta}^\alpha(n-k)+\mathcal{O}(\Gamma_{\beta}^\alpha(n-R)),\;\;\text{for all} \;R\in\mathbb{N}_0,\] 

which gives the desired result. Also, from the equations above we see that $\big(\mathcal{A}_\beta^\alpha f\big)(x)=\big(\mathcal{A}_{\beta+m}^\alpha \displaystyle\frac{f(x)}{x^m}\big)(x)$.

\end{proof}

The next two theorems will be used later in the thesis, the proofs however are lengthy and require many lemmas, and shall thereby be omitted. The reader can refer to \cite{michi} for the complete treatment.

\begin{thm}[\cite{michi}, Prop 4.3.1] \label{derivation}
Let $\alpha,\beta\in \mathbb{R}$, with $\alpha>0$. The linear map $\mathcal{A}_\beta^\alpha$ is a derivation over the ring $\mathbb{R}[[x]]_{\beta}^\alpha$, that is 
\[(\mathcal{A}_\beta^\alpha(f\cdot g))(x)=f(x)(\mathcal{A}_\beta^\alpha g)(x)+g(x)(\mathcal{A}_\beta^\alpha f)(x), \] for all $f,g\in \mathbb{R}[[x]]_{\beta}^\alpha $.
\end{thm}

More interestingly, the next theorem serves as a powerful tool for our purposes. For notation, we set $\mathrm{Diff_{id}}(\mathbb{R},0)=(\{g\in\mathbb{R}[[x]]:g_0=0, g_1=1\},\circ)$, the group of formal diffeomorphisms tangent to the identity, under composition of maps. Similarly, we set  $\mathrm{Diff_{id}}(\mathbb{R},0)_\beta^\alpha=(\{g\in\mathbb{R}[[x]]_\beta^\alpha:g_0=0, g_1=1\},\circ)$ (easily checked to be a monoid).

\begin{thm}[\cite{michi}, Th. 4.4.2]\label{chaintheorem}
Let $\alpha,\beta\in \mathbb{R}$, with $\alpha>0$. Then $\mathrm{Diff_{id}}(\mathbb{R},0)_\beta^\alpha$ is a subgroup of $\mathrm{Diff_{id}}(\mathbb{R},0)$; moreover, for any $f\in\mathbb{R}[[x]]_\beta^\alpha$ and $g\in\mathrm{Diff_{id}}(\mathbb{R},0)_\beta^\alpha$ the following statements  hold:
\begin{enumerate}
    \item $f\circ g$ and $g^{-1}$ are again elements in $\mathbb{R}[[x]]_\beta^\alpha$.
    \item The derivation $ \mathcal{A}_\beta^\alpha$ satisfies a chain rule, namely 
    \begin{equation}\label{chain}
        (\mathcal{A}_\beta^\alpha(f\circ g))(x)=f'(g(x))(\mathcal{A}_\beta^\alpha g)(x)+\bigg(\displaystyle\frac{x}{g(x)}\bigg)^\beta e^{\frac{g(x)-x}{\alpha x g(x)}} (\mathcal{A}_\beta^\alpha f)(g(x)), and
    \end{equation}
     \begin{equation}\label{chain1}
        (\mathcal{A}_\beta^\alpha g^{-1})(x)=-(g^{-1})'(x)\bigg(\displaystyle\frac{x}{g^{-1}(x)}\bigg)^\beta e^{\frac{g^{-1}(x)-x}{\alpha x g^{-1}(x)}}(\mathcal{A}_\beta^\alpha g)(g^{-1}(x)).
    \end{equation}
\end{enumerate}
\end{thm}

It is worth mentioning here that this theorem offers more flexibility than the result by E. Bender in \cite{benderalone}.


\section{The Hopf Algebra of Feynman Diagrams}

This section is a quick review of the algebraic treatment of renormalization in terms of Hopf algebras. The definitions in this section will be needed to give sense of some of the expressions that our results are related to. We have seen what  renormalization is about analytically in the overview of the work by Bogoliubov, Parasiuk, Hepp, and Zimmermann around 1960's. Almost five decades later, D. Kreimer \cite{kreimerr} showed that the BPHZ scheme is captured by Hopf algebras and the recursive definition of the antipode. In this approach, the Feynman diagrams are used to define a connected graded and commutative Hopf algebra. For an extensive treatment of Hopf algebras see \cite{sweedler}, and see \cite{kurusch,conneskreimer,kreimerr,kreimerrr, manchonhopf, renorm} for a spectrum of results and developments of the approach in renormalization in QFT.

\subsection{Basic Definitions of Hopf Algebras}
We let $\mathbb{K}$ be an infinite  field (characteristic 0). The unit of an algebra will be treated as a map in the sense of category theory. $\mathcal{L}(V,W)$ will mean the group of $\mathbb{K}$-linear maps from the $\mathbb{K}$-vector space $V$ to the $\mathbb{K}$-vector space $W$.

\begin{dfn}[Associative unital algebra]\label{aslg}
An \textit{ associative unital $\mathbb{K}$-algebra} $(A,m,\mathbb{I})$ is a $\mathbb{K}$-vector space $A$ together with two linear maps $m:A\otimes A\rightarrow A$ (product) and $\mathbb{I}:\mathbb{K}\rightarrow A$ (unit) such that :
\begin{align*}
    m\circ(\text{id}\otimes m)&=m\circ(m\otimes\text{id})\\
    m\circ(\mathbb{I}\otimes\text{id})&=m\circ(\text{id}\otimes \mathbb{I}).
\end{align*}
Furthermore, the algebra is said to be \textit{commutative} if $m=m\circ\tau$, where $\tau$ is the twist map $a\otimes b\mapsto b\otimes a$.\end{dfn}

The image $\mathbb{I}(1)$ of $1$ under the unit map $\mathbb{I}$ will often be denoted also by $\mathbb{I}$ with no confusion.

In terms of commutative diagrams, $A$ is an associative unital algebra if the following diagrams commute

\begin{center}
\begin{minipage}{0.5 \textwidth}
\xymatrixcolsep{5pc}\xymatrix{
A\otimes A\otimes A \ar[d]_{\text{id}\;\otimes\; m} \ar[r]^{m\;\otimes\;\text{id}} & A\otimes A \ar[d]^m \\
A\otimes A \ar[r]_m          & A }\end{minipage}
\begin{minipage}{0.5 \textwidth}\quad
\xymatrixcolsep{5pc}\xymatrix{
\mathbb{K}\otimes A\ar[rd]_{\cong} \ar[r]^{\mathbb{I}\;\otimes\;\text{id}} & A\otimes A \ar[d]^m &A\otimes\mathbb{K} \ar[l]_{\text{id}\;\otimes\; \mathbb{I}}\ar[ld]^{\cong}\\
&A& }
\end{minipage}
\end{center}

The categorical dual then becomes

\begin{dfn}[Coassociative counital coalgebra]\label{coaslg}
A \textit{ coassociative counital $\mathbb{K}$-coalgebra} $(C,\Delta,\hat{\mathbb{I}})$ is a $\mathbb{K}$-vector space $C$ together with two linear maps $\Delta:C\rightarrow C\otimes C$ (coproduct) and $\hat{\mathbb{I}}:C\rightarrow \mathbb{K}$ (counit) such that :
\begin{align*}
    (\text{id}\otimes \Delta)\circ\Delta&=(\Delta\otimes\text{id})\circ \Delta\\
    (\hat{\mathbb{I}}\otimes\text{id})\circ\Delta&=(\text{id}\otimes \hat{\mathbb{I}})\circ\Delta.
\end{align*}
Furthermore, the coalgebra is said to be \textit{cocommutative} if $\Delta=\tau\circ\Delta$.\end{dfn}

In terms of commutative diagrams, $A$ is an associative unital algebra if the following diagrams commute

\begin{center}
\begin{minipage}{0.5 \textwidth}
\xymatrixcolsep{5pc}\xymatrix{
C\otimes C\otimes C  & C\otimes C \ar[l]_{\Delta\;\otimes\;\text{id}}  \\
C\otimes C \ar[u]^{\text{id}\;\otimes\; \Delta}         & C\ar[l]^\Delta \ar[u]_\Delta }\end{minipage}
\begin{minipage}{0.5 \textwidth}\quad
\xymatrixcolsep{5pc}\xymatrix{
\mathbb{K}\otimes C  & C\otimes C \ar[l]_{\hat{\mathbb{I}}\;\otimes\;\text{id}}\ar[r]^{\text{id}\;\otimes\; \hat{\mathbb{I}}} &C\otimes\mathbb{K} \\
&C\ar[ur]_{\cong}\ar[u]_\Delta\ar[ul]^{\cong}& }
\end{minipage}
\end{center}

\textit{Sweedler's } notation for the coproduct is often useful: $\Delta (x)=\sum_x\;x'\otimes x''$.

\begin{dfn}[Algebra morphism]
Let $(A,m_A,\mathbb{I}_A)$ and $(B,m_B,\mathbb{I}_B)$ be two associative unital $\mathbb{K}$-algebras. A $\mathbb{K}$-linear map $\varphi:A\longrightarrow B$ is an \textit{algebra morphism}  if 
\begin{align*}
    \varphi\circ\mathbb{I}_A&=\mathbb{I}_B, \text{\;and\;}\\
    \varphi\circ m_A&=m_B\circ(\varphi\otimes \varphi). 
\end{align*}
\end{dfn}

Dually, one defines

\begin{dfn}[Coalgebra morphism]
Let $(C,\Delta_C,\hat{\mathbb{I}}_C)$ and $(D,\Delta_D,\hat{\mathbb{I}}_D)$ be two coassociative counital $\mathbb{K}$-coalgebras. A $\mathbb{K}$-linear map $\psi:C\longrightarrow D$ is a  \textit{coalgebra morphism}  if 
\begin{align*}
    \hat{\mathbb{I}}_D\circ\psi&=\hat{\mathbb{I}}_C, \text{\;and\;}\\
    \Delta_D\circ\psi&=(\psi\otimes \psi)\circ\Delta_C. 
\end{align*}
\end{dfn}

\begin{dfn}[Bialgebras]
A \textit{$\mathbb{K}$-bialgebra} $(B,m,\mathbb{I},\Delta, \hat{\mathbb{I}})$ is a $\mathbb{K}$-vector space such that $(B,m,\mathbb{I})$ is an algebra and $(B,\Delta, \hat{\mathbb{I}})$ is a coalgebra and that the two structures are compatible in the sense that  $m$ (and $\mathbb{I}$) is a coalgebra morphism and $\Delta$ (and  $\hat{\mathbb{I}}$) is an algebra morphism.

Note that only one of the compatibility conditions is enough; one can verify that
$m$ is a coalgebra morphism if and only if $\Delta$ is an algebra morphism.
\end{dfn}

Note that in a bialgebra it must be that $ \hat{\mathbb{I}}(\mathbb{I})=1$ and vanishes for all other elements.

\begin{dfn}[Hopf algebras and the antipode]\label{hopf}
A \textit{Hopf algebra}  $(\mathcal{H},m,\mathbb{I},\Delta, \hat{\mathbb{I}},S)$ is a $\mathbb{K}$-bialgebra $(\mathcal{H},m,\mathbb{I},\Delta, \hat{\mathbb{I}})$
together with a linear map $S:\mathcal{H}\rightarrow \mathcal{H}$ such that  

\[m\circ(S\otimes\text{id})\circ\Delta=\mathbb{I}\circ\hat{\mathbb{I}}=m\circ(\text{id}\otimes S)\circ \Delta.\]

The map $S$ is called the \textit{antipode} of the Hopf algebra.
\end{dfn}

Diagrammatically this is equivalent to the following diagram being commutative:

\begin{center}
\begin{minipage}{0.5 \textwidth}
\xymatrixcolsep{5pc}\xymatrix{
\mathcal{H}\otimes \mathcal{H}\ar[r]^{S\;\otimes\;\text{id}}& 
\mathcal{H}\otimes \mathcal{H}\ar[dr]^m& \\
     \mathcal{H} \ar[d]_\Delta \ar[u]^\Delta\ar[r]^{\hat{\mathbb{I}}} &\mathbb{K}\ar[r]^{\mathbb{I}} &            \mathcal{H}\\
\mathcal{H}\otimes \mathcal{H}\ar[r]^{\text{id}\;\otimes\;S}& \mathcal{H}\otimes\mathcal{H}\ar[ur]_m
}
\end{minipage}

\end{center}

\begin{dfn}[Convolution Product]\label{convo}
Let $f,g$ be two linear maps in $\mathcal{L}(\mathcal{H},\mathcal{H})$. Then their \textit{convolution product} is defined as 
\[f\ast g := m\circ(f\otimes g)\circ \Delta.\]

This product gives again a linear map on $\mathcal{H}$. It can be shown that $(\mathcal{L}(\mathcal{H},\mathcal{H}),\ast,\; \mathbb{I}\circ\hat{\mathbb{I}})$ becomes an algebra. Moreover, $f\circ S$ is the inverse of $f$ with respect to the convolution product, and in that sense, the antipode $S$ may be thought of as the $\ast$-inverse of the identity map $\text{id}_\mathcal{H}$.
\end{dfn}

\subsubsection{Filtration and Connectedness of Hopf Algebras}

\begin{dfn}[Gradedness and Connectedness]\label{Aug}
A Hopf algebra $\mathcal{H}$ is said to be \textit{graded} ($\mathbb{Z}_{\geq0}$-graded to be precise) if it decomposes into a direct sum 
$\mathcal{H}=\oplus_{n=0}^\infty\mathcal{H}_n$, such that

\begin{align*}
m(\mathcal{H}_n\otimes\mathcal{H}_m)&\subseteq \mathcal{H}_{n+m},\\  
\Delta(\mathcal{H}_n))&\subseteq \oplus_{k=0}^n\mathcal{H}_{k}\otimes\mathcal{H}_{n-k},\\
S(\mathcal{H}_n)&\subseteq \mathcal{H}_n.
\end{align*}
If, in addition, $\mathcal{H}_0\cong \mathbb{K}$, the Hopf algebra is said to be \textit{connected}.

\end{dfn}
Given a graded Hopf algebra as above, one finds that $\text{ker}\; \hat{\mathbb{I}}=\text{Aug}\mathcal{H}:=\oplus_{n\geq1}^\infty\mathcal{H}_{n}$, called the \textit{augmentation ideal}.

\begin{dfn}[Filtration]
A Hopf algebra $\mathcal{H}$ is \textit{filtered} if there exists a tower of subspaces $\mathcal{H}^n\subseteq \mathcal{H}^{n+1}, n\in \mathbb{N}$, such that 

\begin{align*}
\mathcal{H}&=\sum_{n=0}^\infty\mathcal{H}^n,\\
m(\mathcal{H}^n\otimes\mathcal{H}^m)&\subseteq \mathcal{H}^{n+m},\\  
\Delta(\mathcal{H}^n))&\subseteq \sum_{k=0}^n\mathcal{H}^{k}\otimes\mathcal{H}^{n-k},\\
S(\mathcal{H}^n)&\subseteq \mathcal{H}^n.
\end{align*}
Note that every graduation implies a filtration by taking $\mathcal{H}^n=\oplus_{k=0}^n\mathcal{H}_k$.
\end{dfn}

\begin{dfn}[Primitive and group-like elements]
An element $x\in\mathcal{H}$ is \textit{primitive} if $\Delta(x)=\mathbb{I}\otimes x+x\otimes {\mathbb{I}}$. An element $x$ is \textit{group-like} if $\Delta(x)=x\otimes x$.
\end{dfn}

\subsection{Physical Theories as Combinatorial Classes of Graphs}

Now we are ready to define the renormalization Hopf algebras of Feynman diagrams, but first let us emphasize the combinatorial rephrasing of the physical setup already seen in the previous sections.

All of the enumerative aspects of QFT considered in this thesis are about Feynman diagrams. Feynman diagrams and their Hopf algebras will be key ingredients in the later chapters, although not explicitly affecting the combinatorial problems we consider. We will proceed by defining  combinatorial QFT theories, Feynman graphs, and Feynman rules. Then we will recover some of the concepts of renormalization. Recommended references for similar treatments are \cite{karenbook,manchonhopf, michi}.

The building block for Feynman graphs is going to be \textit{half edges}. An edge is intuitively understood to be formed from two half edges.

\begin{dfn}
A graph (or diagram) $G$ is a set of half edges for which there is 
\begin{enumerate}
    \item a partition $V(G)$ into disjoint classes of half edges, a class $v\in V(G)$ will be called a \textit{vertex};
    \item a collection $E(G)$ of disjoint pairs of half edges. $E(G)$ will be called the set of \textit{internal edges};
    \item half edges that are not occurring in any of the pairs in $E(G)$ will be called \textit{external edges} or \textit{external legs}.
\end{enumerate}

\end{dfn}

The size of a graph will be the size of its set of half edges. Half edges can be labelled or unlabelled, and sometimes we will use many types of half edges to represent a certain physical theory. We will be concerned at some point with graphs which have a prescribed set of external legs. The \textit{loop number} of a graph is the dimension of its cycle space, or in other words the number of independent cycles. There exists, in any graph, a family of independent cycles with each cycle having an edge not occurring in any of the other cycles in the family. The size of the largest such family is the number of independent cycles in the graph. Such a family of cycles can be obtained by starting with a spanning tree and reading off the new cycle created by adding one of the edges, one edge at a time (and with removing any edge added earlier). The loop number will be very important in our later considerations and will express some sort of size for diagrams with a prescribed scheme of external legs.

By $\text{Aut}(G)$ we mean the group of automorphisms (self isomorphisms) of the graph $G$. In the perturbative expansions that we will see, a Feynman diagram  will have a \textit{symmetry} factor of $1/|\text{Aut}(G)|$ (it is more common to write it as $1/\text{Sym}(G)$. We will usually work with unlabelled graphs, nevertheless, the symmetry factor will allow us to use the exponential relation between connected and disconnected objects.

The good thing about many of the aspects of quantum field theory is that they can be transformed into purely combinatorial and enumerative problems. 

\begin{dfn}\label{power counting weight}
A \textit{combinatorial physical theory} consists of 
\begin{enumerate}
    \item a dimension of spacetime (nonnegative integer);
    \item a number of half edge types, and a set of pairs of half edge types, with each pair representing an admissible edge type in the theory (note that the half edge types in one pair are not necessarily distinct nor identical);
    
    \item a collection of multisets of half edge types to define the options for a vertex in the theory;
    
    \item an integer weight for each edge or vertex type, called a \textit{power counting weight}.
    
\end{enumerate}
\end{dfn}

Thus, a graph in a certain theory $T$ will be a graph whose edges are of the types formed by the admissible pairs of $T$, and each of whose vertices is incident to an admissible multiset of half edges. Note that even oriented and unoriented edges can be formed this way: oriented edges arise from an admissible pair of half edges in which the two types are different, whereas unoriented ones arise from pairs with the two types identical.\\

\begin{exm}
\begin{enumerate}
    \item \textbf{QED: Quantum electrodynamics.} In QED there are 3 half edge types: a half photon, a front half fermion, and a back half fermion. The admissible combinations of half edges to form edges are: (1) a pair  of two half photons to give a photon edge, drawn as a wiggly line \raisebox{-0.1cm}{\includegraphics[scale=0.3]{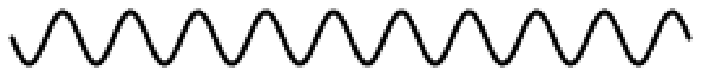}}, with power counting weight 2; and (2) a pair consisting of a front and back halves fermion to give a directed fermion edge \includegraphics[scale=0.8]{Figures/fermionedge.eps}, with power counting weight 1. There is one type of a vertex, namely, a vertex is 3-valent and is incident to one of each half edge type, with power counting weight 0. The spacetime dimension is taken to be 4.
    
    \item \textbf{Yukawa theory:} This theory also has 3 types of half edges: a half meson, a front half fermion, and a back half fermion. The admissible edges are: (1) a meson edge formed by two half mesons (front and back), drawn as \raisebox{-0.1cm}{\includegraphics[scale=0.3]{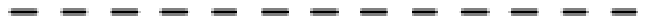}}, and has power counting weight 2; and (2) a pair of a front and back halves fermion to give a directed fermion edge \includegraphics[scale=0.8]{Figures/fermionedge.eps}, with weight 1. Just as in QED, there is one type of vertices, namely, a vertex is 3-valent and is incident to one of each half edge type, with power counting weight 0. The spacetime dimension is taken to be 4. The difference from QED lies in the Feynman rules.
\end{enumerate}
\end{exm}

As mentioned earlier, the significance of quantum field theory is the ability to describe how particles interact and scatter. In an idealized experiment some particles are sent in, they interact and scatter, and then the outcomes are detected. This picture can be visualized as a diagram in which the edges describe propagating particles. The idea then is that, on an atomic scale, we never know what exactly happened and every possible interaction is assigned a probability \textit{scattering amplitude}. This amounts into a weighted sum, known as a \textit{perturbative expansion}.  The probabilities in the theory are computed through what is known as a \textit{Feynman integral}. These integrals encountered by physicists are  often divergent and have to undergo renormalization to retrieve useful information.  As we saw before, Feynman graphs encode these complicated integrals, and the rules for this encoding  in a given QFT are known as \textit{Feynman rules}.

\begin{dfn}[Feynman Graphs]
A Feynman graph in a theory $T$ is combinatorially a graph structure in which edges fall into certain types and vertices are subject to conditions (pertinent to $T$) on the number of edges of a certain types attached to it. A Feynman graph represents an integral through the Feynman rules of the theory, which assigns an integrand factor contribution to every internal edge or vertex. The power counting weights give the degree of an integrated momentum variable. \end{dfn}

\begin{exm}
In the following example  (Figure \ref{F}) the integral is a divergent Feynman integral and its corresponding Feynman diagram:

\centering
\begin{figure}[h!]
   \centering
   \[   \begin{minipage}[h]{0.3\linewidth}
	\vspace{0pt}
	\includegraphics[width=\linewidth]{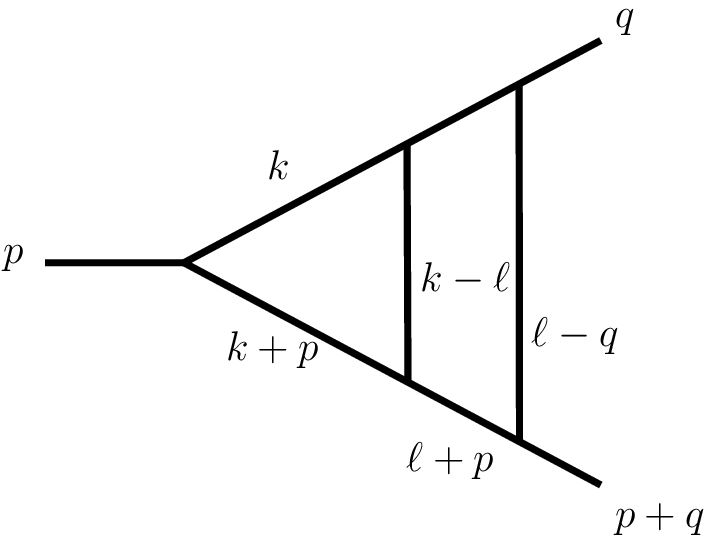}
    \end{minipage}
     = 
   \begin{minipage}[h]{0.5\linewidth}
	\vspace{0pt}
	$\displaystyle \int \int\frac{d^D\ell}{\ell^2(\ell-q)^2(\ell+p)^2} \frac{d^Dk}{k^2(k-\ell)^2(k+p)^2}.$
    \end{minipage}
   \]
   \caption{The Feynman integral of a Feynman graph}\label{F}
\end{figure}

\end{exm}

\subsection{Characters and Cocycles}

\begin{dfn}[Characters]
Let $\mathcal{H}$ be a connected bialgebra and $(A,\cdot,1_A)$ be a $\mathbb{K}$-algebra. A \textit{character} from $\mathcal{H}$ to $A$ is defined to be an algebra morphism with the extra property that $\phi(\mathbb{I})=1_A$. The set of all characters from $\mathcal{H}$ to $A$ is denoted by $G_A^\mathcal{H}$. Further, if $A$ is commutative, $G_A^\mathcal{H}$ becomes a group under convolution product \cite{panzer}. The inverses are denoted $\varphi^{\ast -1}:=\varphi\circ S$.

\end{dfn}

\begin{dfn}
Let $\mathcal{H}$ be a connected bialgebra and $A$ be a commutative algebra that can be written as a direct sum of two vector spaces. A \textit{Birkhoff decomposition} of a character $\phi$ is a pair of characters $\phi_+,\phi_-\in G_A^\mathcal{H}$ such that \[\phi=\phi_-^{\ast -1}\ast\phi_+ \;\text{\;and\;}\; \phi_{\pm}(\mathrm{ker}\hat{\mathbb{I}})\subseteq A_{\pm}.\] 
\end{dfn}

In \cite{conneskreigeom} it was shown that dimensional regularization (viewing the integral over dimension $D-2\epsilon$ and expanding in $\epsilon$) can be studied in terms of characters into the algebra of meromorphic functions in $\epsilon$.

\begin{thm}[\cite{manchonhopf}]\label{birkhoff}

Let $\mathcal{H}$  be a connected filtered Hopf algebra, and let $G_A^\mathcal{H}$ be the group of characters with the convolution product. Then any character $\varphi\in G_A^\mathcal{H}$ has a unique Birkhoff decomposition 
\[\varphi=\varphi^{\ast -1}_-\ast\varphi_+,\] where $\varphi_- ,\varphi_+\in G_A^\mathcal{H}$, with $\varphi_-$ mapping the augmentation ideal into $A_-$, and  with $\varphi_+$ mapping $\mathcal{H}$ into $A_+$. The characters naturally satisfy $\varphi_-(\mathbb{I})=1_A=\varphi_+(\mathbb{I})$ and are defined recursively over the augmentation ideal as 

\begin{align*}
    \varphi_-(x)&=-\pi\big(\varphi(x)+\sum_x\varphi_-(x')\varphi(x'')\big), \;\text{and}\\
    \varphi_+(x)&=(\mathrm{id}-\pi)\big(\varphi(x)+\sum_x\varphi_-(x')\varphi(x'')\big),\end{align*}
    where $\pi$ is the projection of $A$ onto $A_-$, and the sum is making use of Sweedler's notation for the coproduct.

\end{thm}

\begin{dfn}[Bogoliubov map]\label{hopf bog}
The \textit{Bogoliubov map} is  the map $b:G\longrightarrow \text{Hom}(\mathcal{H},A)$ defined recursively by

\[b(\varphi)(x)=\varphi(x)+\sum_x\varphi_-(x')\varphi(x'').\]

In particular, the decomposition in Theorem \ref{birkhoff} is  now seen via the Bogoliubov map as 
\[\varphi_-=-\pi\circ b(\varphi)\;,\qquad\text{and}\qquad\varphi_+=(\mathrm{id}-\pi)\circ b(\varphi).\]
\end{dfn}

 Before starting the next part, it must be noted that this section does not give a full account of the Hopf-algebraic treatment of renormalization. We are only interested in defining the expressions that we will encounter in our problems. The reader can refer to \cite{sweedler} for an in depth account on Hopf algebras. The Hopf algebra of Feynman graphs is also surveyed in the review article of D. Manchon \cite{manchonhopf}.

\subsection{The Hopf Algebra of divergent 1PI Diagrams}\label{hopfalg1PIsection}

Let $T$ be a fixed combinatorial physical theory in the sense of the previous section, and consider the $\mathbb{Q}$-vector space $\mathcal{H}$ generated by the set of disjoint unions of divergent $1PI$ Feynman graphs in the theory $T$, including the empty graph which we denote by $\mathbb{I}$.

We can define a multiplication $m$ on $\mathcal{H}$ to be taking the disjoint union, and the unit is the empty graph $\mathbb{I}$, this makes $(\mathcal{H},m,\mathbb{I})$ a commutative associative algebra.

Now we need to define a compatible coalgebra structure for $\mathcal{H}$.

In the next just note that the \textit{residue} of a Feynman graph is simply the graph obtained if all internal edges were contracted. For the sake of a precise general definition of contraction in the new terms we have

\begin{dfn}[Contraction of a Subgraph]\label{contractions}
Let $\Gamma$ be a Feynman graph in a theory $T$, and let $\gamma\subseteq \Gamma$ be a subgraph each of whose connected components is $1PI$ and divergent. The \textit{contraction graph} $\Gamma/\gamma$ is constructed as follows:
\begin{enumerate}
    \item A component of $\gamma$ with a vertex residue (external leg structure)  is contracted in $\Gamma$ into a vertex of the same type as the residue.
    
    \item A component of $\gamma$ with an edge residue (external leg structure)  is contracted in $\Gamma$ into an edge of the same type as the residue.
\end{enumerate}
\end{dfn}

The superficial degree of divergence $\omega(\Gamma)$ of a graph $\Gamma$ is defined in many references (\cite{karenbook, renorm, AliThesis}) in which it is also shown that $\omega(\Gamma)=D\;\ell(\Gamma)-\sum_a w(a)$, where the sum is over all the power counting weights (Definition \ref{power counting weight})  of vertices and internal edges in $\Gamma$ determined by the  QFT theory considered, and where $\ell(\Gamma)$ is the number of loops in $\Gamma$.

\begin{dfn}[Subdivergence]\label{subdiv}
A subgraph $\gamma$ of $\Gamma$ with divergent $1PI$ connected components   is called a \textit{subdivergence}. 
\end{dfn}

Then we define the coproduct as

\begin{dfn}\label{coprod}
The coproduct $\Delta:\mathcal{H}\longrightarrow\mathcal{H}\otimes\mathcal{H}$ is defined for a connected Feynman graph $\Gamma$ to be 
\[\Delta(\Gamma)=\underset{1PI \;\text{subgraphs}}{\underset{\gamma\;\text{product of divergent }}{\underset{\gamma\subseteq\Gamma}{\sum}}}\gamma\otimes\Gamma/\gamma\]

and extended as an algebra morphism.
\end{dfn}

Note that since we are considering graphs that are themselves divergent, the coproduct sum for any element in $\mathcal{H}$ will  always start as 

\[\Delta(\Gamma)=\mathbb{I}\otimes\Gamma+\Gamma\otimes\mathbb{I}+\widetilde{\Delta}(\Gamma).\]
 
The part $\widetilde{\Delta}(\Gamma)$ of the coproduct is called the \textit{reduced coproduct}. 

\begin{dfn}[Primitive Elements]\label{primitivediags}
An element $\Gamma\in\mathcal{H}$ is said to be \textit{primitive } if 
$\widetilde{\Delta}(\Gamma)=0$. That is, ${\Delta}(\Gamma)=\mathbb{I}\otimes\Gamma+\Gamma\otimes\mathbb{I}$.
 In particular, a primitive $1PI$ graph $\Gamma$ is a $1PI$ graph that contains no subdivergences in the sense of Definition \ref{subdiv}.\end{dfn}

For example let us calculate the coproduct 

\begin{flalign*}
    &\Delta\bigg(\;\raisebox{-0.5cm}{\includegraphics[scale=0.3]{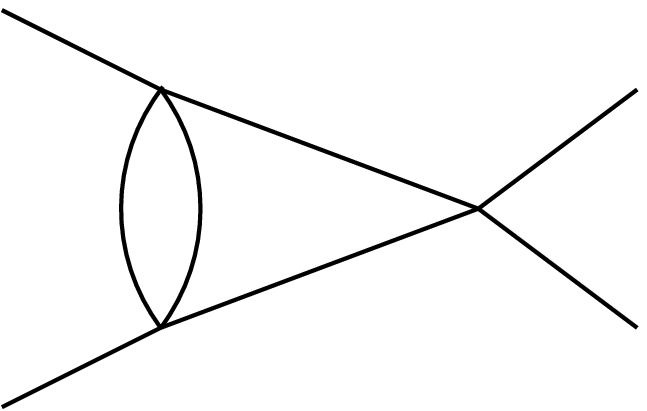}}-\raisebox{-0.30cm}{\includegraphics[scale=0.3]{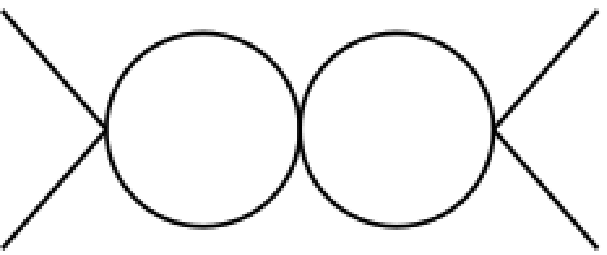}}\;\bigg)\\
    =
    &\;\mathbb{I}\otimes
    \bigg(\;\raisebox{-0.5cm}{\includegraphics[scale=0.3]{Figures/coprodsmall1.eps}}-\raisebox{-0.30cm}{\includegraphics[scale=0.3]{Figures/coproduct3.eps}}\;\bigg)+
  \bigg(\;\raisebox{-0.5cm}{\includegraphics[scale=0.3]{Figures/coprodsmall1.eps}}-\raisebox{-0.30cm}{\includegraphics[scale=0.3]{Figures/coproduct3.eps}}\;\bigg)\otimes\mathbb{I}\;+\qquad\qquad\qquad\qquad\\
    &\;+\;\raisebox{-0.3cm}{\includegraphics[scale=0.3]{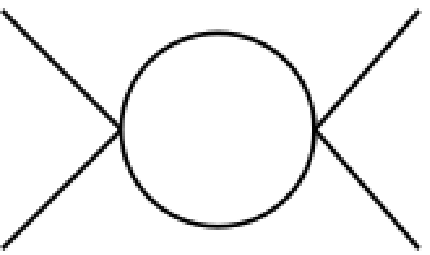}}\;\otimes\;\raisebox{-0.3cm}{\includegraphics[scale=0.3]{Figures/coproduct2.eps}}\;-\;2\;\raisebox{-0.3cm}{\includegraphics[scale=0.3]{Figures/coproduct2.eps}}\;\otimes\;\raisebox{-0.3cm}{\includegraphics[scale=0.3]{Figures/coproduct2.eps}}\\
    =
    &\;\mathbb{I}\otimes\bigg(\;\raisebox{-0.5cm}{\includegraphics[scale=0.3]{Figures/coprodsmall1.eps}}-\raisebox{-0.30cm}{\includegraphics[scale=0.3]{Figures/coproduct3.eps}}\;\bigg)+
  \bigg(\;\raisebox{-0.5cm}{\includegraphics[scale=0.3]{Figures/coprodsmall1.eps}}-\raisebox{-0.30cm}{\includegraphics[scale=0.3]{Figures/coproduct3.eps}}\;\bigg)\otimes\mathbb{I}-
\raisebox{-0.3cm}{\includegraphics[scale=0.3]{Figures/coproduct2.eps}}\otimes\raisebox{-0.3cm}{\includegraphics[scale=0.3]{Figures/coproduct2.eps}}.
\end{flalign*}
   
   Finally, let\;\; $\hat{\mathbb{I}}:\mathcal{H}\longrightarrow\mathbb{Q}$\;\; be the map defined on the empty graph by sending $q\mathbb{I}$ to $q\in\mathbb{Q}$ and sending every other element in $\mathcal{H}$ to zero.
Then it is not hard to prove the following proposition (see \cite{karenbook,cutcut,karenthesis} for a proof)

\begin{prop}
As per the above definitions, $(\mathcal{H},m,\mathbb{I},\Delta,\hat{\mathbb{I}})$ is a bialgebra. Further, if we define a map $S:\mathcal{H}\longrightarrow\mathcal{H}$ recursively by

\begin{align*}
    S(\mathbb{I})&=\mathbb{I}\;, \text{\;and}\\
    S(\Gamma)&=-\Gamma-\underset{1PI \;\text{subgraphs}}{\underset{\gamma\;\text{product of divergent }}{\underset{\mathbb{I}\neq\gamma\neq\Gamma}{\underset{\gamma\subseteq\Gamma}{\sum}}}}\;S(\gamma)\;\Gamma/\gamma,\end{align*}
then $(\mathcal{H},m,\mathbb{I},\Delta,\hat{\mathbb{I}},S)$ becomes a Hopf algebra woth antipode $S$. (Note that the product in the second term is the product $m$ abbreviated). Moreover, the Hopf algebra $\mathcal{H}$ is commutative and is graded by the loop number.
\end{prop}

\begin{rem}\label{core hopf}
In the next section we will broadly see how renormalization is represented in this algebraic context of Hopf algebras. Our job ends with learning the meaning of some of the expressions that will show up again in our problems. It should be noted however that, as expected, this is not the only meaningful appearance of Hopf algebras in quantum field theory. Namely, if the condition of divergence is dropped from the elements summed over in the definitions of the coproduct and the antipode,  we get the so-called  \textit{the core Hopf algebra}, denoted $\mathcal{H}_c$. It turns out that $\mathcal{H}_c$ interplays with Cutkosky cuts in graphs, this is related to the unitarity of the $S$-matrix \cite{cutcut}.
\end{rem}

\subsection{Renormalization in Hopf algebras}

\subsubsection{Feynman Rules and Characters of $\mathcal{H}$:}

Let a theory $T$ be fixed as before, and let $\mathcal{H}$ be the Hopf algebra generated by sets of divergent $1PI$ Feynman graphs in $T$. We start by thinking of Feynman rules as a map $\phi$ that assigns formal integrals to elements in $\mathcal{H}$, and we investigate what conditions should be imposed on $\phi$ to fully interpret the Feynman rules.

For the Feynman rules, we need to satisfy certain criteria:

\begin{enumerate}
    \item The map $\phi$ should be multiplicative on disjoint unions of graphs. Moreover, the map should also have a multiplicative property for bridges. The latter requirement enables us to start defining $\phi$ over $1PI$ diagrams. The leap from all Feynman graphs to $1PI$ graphs is done through the Legendre transform, which has been redefined recently as a purely combinatorial map \cite{kjm2,kjm1}. 
    
    \item The map $\phi$, representing Feynman rules, has also to adapt with the combinatorial Dyson-Schwinger equations. Precisely, it has
    to interplay nicely with the process of \textit{insertion} which we discuss in the next section.
    \end{enumerate}

\textbf{(A)}\label{AAA} All of this was seen to suggest that the Feynman rules are to be represented by a character $\phi\in G^\mathcal{H}_A$, where $A$ is a suitably chosen commutative algebra. The target algebra $A$ is usually taken to be the algebra $\mathbb{C}[L][[z^{-1},z]]$ of Laurent series whose coefficients are polynomials in an energy scale $L$. For example, in \cite{renorm}, $L=\log (q^2/\mu^2)$ where $q$ and $\mu$ are the external momenta and the renormalization scale respectively.

\subsubsection{Rota-Baxter Operators:}

Let $A$ be an algebra as before. An operator (linear map on $A$) $R:A\longrightarrow A$ is said to be a \textit{Rota-Baxter operator} if it satisfies

\[R[ab]\;+\;R[a]\;R[b]\;=\;R\big[R[a]\;b\;+\;a\;R[b]\big],\]
for all $a,b\in A$.

Indeed, it turns out that the truncated Taylor operator  $T^{\omega(\Gamma)}$ in the BPHZ scheme is a Rota-Baxter operator. This relation between renormalization and Rota-Baxter operators has been extensively studied in \cite{kurusch, kurusch2}.

\textbf{(B)} \label{BBB} In general, a  Rota-Baxter operator will be used to express a map which sends a formal integral to the evaluation of the integral at the subtraction point in the renormalization scheme. In other words, $R$ produces the counterterms. If $\Gamma$ is a divergent graph with no subdivergences, $R\phi(\Gamma)$ will stand for the ill part of the integral $\phi(\Gamma)$.
\\

It remains to setup a technology for dealing  with subdivergences recursively.

Define a linear map $S^{\phi}_R:\mathcal{H}\longrightarrow A$
by $S^{\phi}_R(\mathbb{I})=1_A$ and 

\begin{equation}\label{S_R}
    S^{\phi}_R(\Gamma)=-R(\phi(\Gamma))-\underset{1PI \;\text{subgraphs}}{\underset{\gamma\;\text{product of divergent }}{\underset{\mathbb{I}\neq\gamma\subsetneq\Gamma}{\sum}}}S^{\phi}_R(\gamma)\; R(\phi(\Gamma/\gamma)),
\end{equation}
 and extended to all of $\mathcal{H}$ as a morphism of algebras. 

\textbf{(C)} Then the \textit{renormalized} Feynman rules are defined to be 
\begin{equation}
    \phi_R=S^{\phi}_R\ast\phi.
\end{equation}

It can be shown that $\phi_R(x)=(S^{\phi}_R\ast\phi)(x)=(\mathrm{id}_A-R)b(\phi)(x),$
where $b$ is the Bogoliubov map defined before (Definition \ref{hopf bog}) \cite{renorm}.
\\

By (A), (B), and (C), the conclusion is that the approach of renormalization is as follows:
(1) We express Feynman graphs in a graded Hopf algebra $\mathcal{H}$, and interpret the Feynman rules as characters from $\mathcal{H}$ to some commutative algebra $A$. (2) A renormalization scheme is determined via a Rota-Baxter operator on $A$, this also determines a Birkhoff decomposition $A=A_-\oplus A_+$ into two subalgebras. (3) The renormalized Feynman rules are obtained through the coproduct and the map $S^{\phi}_R$. For explicit examples and applications of this approach the reader can refer to \cite{renorm,karenbook,manchonhopf}.

\subsection{Combinatorics of Dyson-Schwinger Equations}\label{DSE}

\subsubsection{Insertions }
Definition \ref{contractions} introduces the notion of contracting a subgraph within a bigger graph. One can think of a reverse operation in terms of \textit{inserting} a graph $\gamma$ into another graph $\Gamma$ as a subgraph, in one of the potential positions (\textit{insertion places}) in $\Gamma$ that can host $\gamma$. An insertion place has to be compatible with the external leg structure of the graph being inserted.

\begin{dfn}[Insertion]\label{insertion}
Let $\gamma$ be a Feynman graph with external leg structure $r=\mathrm{res}(\gamma)$. Let $\Gamma$ be a Feynman graph with a vertex or an internal edge of the same type as $r$. 
\begin{enumerate}
    \item If $r$ is of edge type, and $e$ is an internal edge in $\Gamma$ of the same type,  then we can \textit{insert} $\gamma$ into $\Gamma$ as follows:
    
    Break the edge $e$ into two half edges, each of which is identified with one of the two compatible external legs of $\gamma$. 
    
    \item If $r$ is of vertex type, and $v$ is a vertex of the same type in $\Gamma$, then we can \textit{insert} $\gamma$ into $\Gamma$ as follows:
    Break every edge incident to $v$, and, in a compatible way, which may not be unique, attach the external legs of $r$ to the resulting half edges in $\Gamma-v$.
\end{enumerate} 

The places $e$ or $v$ in the above scenarios are called \textit{insertion places}. Notice that the way to insert $\gamma$ into $\Gamma$ at a certain insertion place is not unique and depends on the symmetries of the graphs.
\end{dfn}

We wish now to define an operator $B_+^\Gamma$ that inserts graphs into a fixed graph $\Gamma$. 

\begin{dfn}[\cite{karenbook}]
For a connected $1PI$ Feynman graph $\Gamma$ we define
\[B_+^\Gamma(X)=\underset{G\;\text{1PI\;graph}}{\sum} \displaystyle\frac{\text{bij}(\Gamma,X,G)}{|X|_*}\frac{1}{\text{maxf}(G)}\frac{1}{(\Gamma|G)}\;G,\]

where 
\begin{enumerate}
    \item $\text{maxf}(G)$ is the number of insertion trees corresponding to $G$,
    \item $|X|_*$ is the number of distinct graphs obtained from permuting the external legs in $X$,
    \item $\text{bij}(\Gamma,X,G)$ is the number of bijections of the external legs of $X$ which have an insertion place in $\Gamma$ so that the insertion gives $G$.
    \item $(\Gamma|G)$ is the number of insertion places for $X$ in $\Gamma$.
\end{enumerate}
\end{dfn}

\begin{rem}\label{mercy}
See \cite{kreimerB} Theorem 4 for a justification of this definition. In the case of trees, the operation $B_+(T_1\cdots T_m)$ takes the rooted trees $T_1,\cdots,T_m$ and attach all of their roots as children of a new added root, getting a single rooted tree (the name \textit{grafting} operator makes sense in this case).  
\end{rem}

In the case of rooted trees described in the remark above, if $\mathcal{H}_C$ denotes the Connes-Kreimer Hopf algebra of rooted trees \cite{conneskreigeom,karenbook}, then the grafting operator $B_+$ is characterized as being a \textit{Hochschild  $1$-cocycle}, that is:  \[\Delta\circ B_+(t)=(\mathrm{id}\otimes B_+)\circ\Delta(t)+B_+(t)\otimes \mathbb{I}.\]

 This property will be highlighted in the next section as it is crucial to the algebraic reconstruction of renormalization in the approach pioneered by D. Kreimer and his collaborators. For more about this algebraic treatment and concepts see the original paper by D. Kreimer and A. Connes \cite{conneskreigeom} or \cite{karenbook}.

\begin{rem}\label{B+cocycle}
In order to get a $1$-cocycle from the operators $B_+^\gamma$ it turns out that we can not work with individual primitive graphs, rather, we should sum over all primitive graphs of a given loop number \cite{karenbook,kreimerB}
\end{rem}

To express the Dyson-Schwinger equations in terms of the operators $B_+^\Gamma$ we need to know more about the number of insertion places in a given graph.

Let us assume that the combinatorial theory we are considering now has only one vertex type $v$, and let $d$ be the  degree of any such vertex. Also set $n(e)$ to be the number of half edges of type $e$ appearing in the external legs of vertex-type $v$. By definition we set $n(v)=1$.

\begin{prop}[\cite{karenbook}]
Let $\Gamma$ be a $1PI$ graph in a QFT theory of the type described above, that is, the theory has only one vertex type $v$ with $d$ being the degree of such a vertex.  Let $r=\mathrm{res}(\Gamma)$, and $\ell=\ell(\Gamma)$ (the loop number). Also let $n(e)$  be the number of half edges of type $e$ appearing in the external legs of vertex-type $v$. Besides, define $n(v)=1$. Then 

\begin{enumerate}
    \item $\Gamma$ has $\displaystyle\frac{2\ell n(s)}{d-2}$ insertion places for every type $s\neq r$;
    \item If $r$ is vertex-type, then $\Gamma$ has\; $1+\displaystyle\frac{2\ell n(r)}{d-2}$ insertion places for type $r$; and
    \item If $r$ is not vertex-type, then $\Gamma$ has\; $-1+\displaystyle\frac{2\ell n(r)}{d-2}$ insertion places for type $r$.
\end{enumerate}
\end{prop}

\begin{exm}

In QED (quantum electrodynamics) we have only one vertex type, namely \raisebox{-0.67cm}{\includegraphics[scale=0.25]{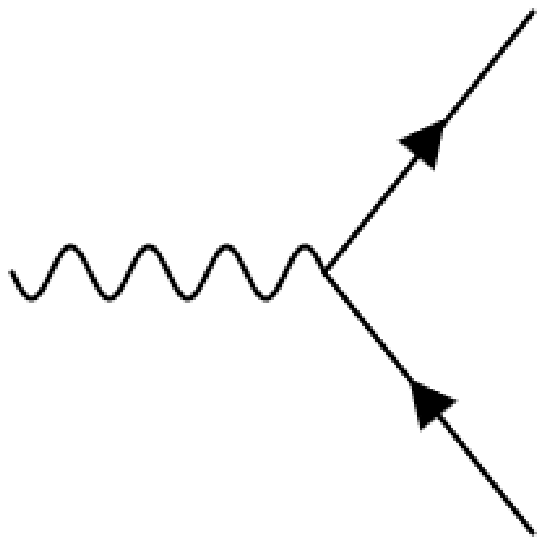}}, and two edge types: a photon edge  \raisebox{-0.1cm}{\includegraphics[scale=0.26]{Figures/photonedge.eps}}, and a fermion edge \;\raisebox{-0.0cm}{\includegraphics[scale=0.7]{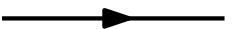}}.

We are going to follow the notation used in \cite{karenbook}, namely

\begin{itemize}
    \item $X^{\text{vertex}}$ is the \textit{vertex series}, whose $n$th coefficient is the sum of all $1PI$ QED diagrams with residue  \raisebox{-0.57cm}{\includegraphics[scale=0.2]{Figures/QEDvertex.eps}} and loop number $n$.
    
    \item $X^{\text{photon}}$ is the \textit{photon edge series}, whose $n$th coefficient is $(-1)\times$\{the sum of all $1PI$ QED diagrams with residue  \raisebox{-0.14cm}{\includegraphics[scale=0.2]{Figures/photonedge.eps}} and loop number $n$\}.
    
    \item $X^{\text{fermion}}$ is the \textit{fermion edge series}, whose $n$th coefficient is $(-1)\times$\{the sum of all $1PI$ QED diagrams with residue  \raisebox{0cm}{\includegraphics[scale=0.48]{Figures/ferm1.eps}} and loop number $n$\}.
\end{itemize}

\begin{rem}
Notice that the negative signs with the edge series arise as we will be actually interested in sequences of such diagrams, and so if $Y$ is the original generating function then we are to get a geometric series $\displaystyle\frac{1}{1-Y}$. Then we use $X=1-Y$.
\end{rem}

\setlength{\parindent}{0.5cm}
Then we have 
\begin{align}
    X^{\text{vertex}}&=\mathbb{I}+ \underset{\text{vertex residue}}{\underset{\gamma \text{\;primitive with}}{\sum}} x^{\ell(\gamma)}
    \;B_+^\gamma\left(\displaystyle\frac{\left(X^{\text{vertex}}\right)^{1+2\ell(\gamma)}}{\left(X^{\text{photon}}\right)^{\ell(\gamma)}\left(X^{\text{fermion}}\right)^{2\ell(\gamma)}}\right),\label{e1}\\
    &\nonumber\\
    X^{\text{photon}}&=\mathbb{I}-  x
    \;B_+^{\raisebox{-0.84cm}{\includegraphics[scale=0.182]{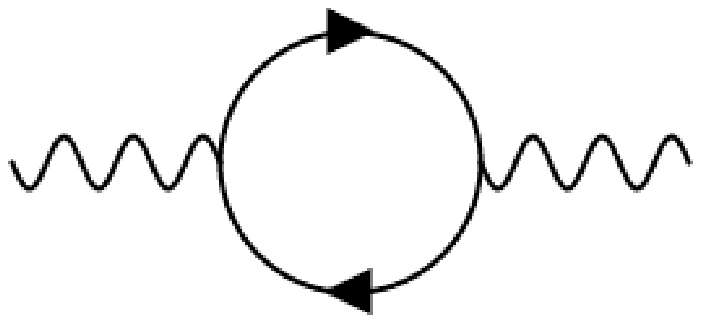}}}
    \left(\displaystyle\frac{\left(X^{\text{vertex}}\right)^2}{\left(X^{\text{fermion}}\right)^2}\right),\label{e2}\\
    &\nonumber\\
    X^{\text{fermion}}&=\mathbb{I}- x
    \;B_+^{\raisebox{-0.57cm}{\includegraphics[scale=0.18]{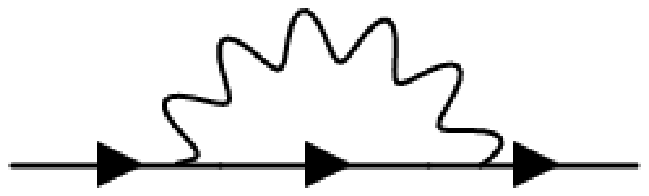}}}
   \left(\displaystyle\frac{\left(X^{\text{vertex}}\right)^2}{X^{\text{photon}}\;\;X^{\text{fermion}}}\right).\label{e3}
\end{align}

These equations are obtained by a direct counting argument. For example, the second equation can be illustrated through Figure{\ref{lloop}}.

\begin{figure}[h]
 \begin{center}

\raisebox{-0.84cm}{\includegraphics[scale=0.6]{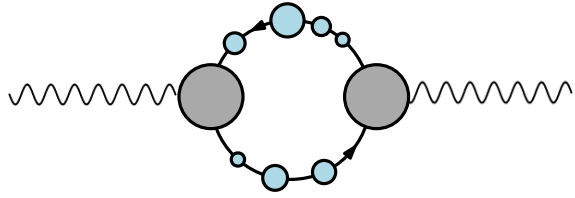}}
\end{center} \caption{Graphs with photon edge residue}\label{lloop}
\end{figure}

In Figure \ref{lloop} the blue bubbles represent the two sequences of fermion-type $1PI$ graphs that can be inserted along the original two fermion edges, hence the $1/\left(X^{\text{fermion}}\right)^2$; whereas the two larger grey bubbles represent insertion of a vertex-type $1PI$ graph and correspond to the $\left(X^{\text{vertex}}\right)^2$. The minus sign follows from the definition of $X^{\text{photon}}$.
\end{exm}

\subsection{The Invariant Charge}\label{invariantch}

If we set $Q=\displaystyle\frac{\left(X^{\text{vertex}}\right)^2}{
\left(X^{\text{photon}}\right)^{}\;\left(X^{\text{fermion}}\right)^2},$
then we can rewrite equations (\ref{e1}), (\ref{e2}), and (\ref{e3}) as

\begin{align}
    X^{\text{vertex}}&=\mathbb{I}+ \underset{\text{vertex residue}}{\underset{\gamma \text{\;primitive with}}{\sum}} x^{\ell(\gamma)}
    \;B_+^\gamma\left(X^{\text{vertex}} Q^{\ell(\gamma)}\right),\label{e11}\\
    X^{\text{photon}}&=\mathbb{I}-  x
    \;B_+^{\raisebox{-1cm}{\includegraphics[scale=0.18]{Figures/loop1.eps}}}
    \left(X^{\text{photon}} Q\right),\label{e22}\\
    X^{\text{fermion}}&=\mathbb{I}- x
    \;B_+^{\raisebox{-0.57cm}{\includegraphics[scale=0.18]{Figures/loop2.eps}}}
   \left(X^{\text{fermion}} Q\right).\label{e33}
\end{align}

The expression $Q$ is to be called the \textit{invariant charge}. In general, for a theory with only one vertex type $v$, the system of Dyson-Schwinger equations takes the form \begin{equation}\label{generaldys}
    X^r=1\pm \sum_k B_+^{\gamma_{r,k}}(X^rQ^k),
\end{equation}
where the sum is over $k$ and over all primitive 1PI diagrams $\gamma_{r,k}$ with loop number $k$ and residue $r$.

\begin{equation}\label{qgeneral} 
    Q=\left(\displaystyle\frac{X^v}{\sqrt{\prod_{e\in v}(X^e)}}\right)^{\frac{2}{d-2}},
\end{equation}
where $d$ is the degree of the vertex-type, and the product is over all half edges making up the vertex. Note that the orientation of a half edge is ignored in counting the types. In some references, the convention for the invariant charge is to be the square root of our definition \cite{michiq}. Also in \cite{michiq} the $X^r$ is defined to be the negative of ours in case $r$ is an edge type.

\bibliographystyle{plain}
\bibliography{References.bib}

\end{document}